\documentclass[12pt,a4paper]{amsart}

\usepackage{a4wide}
\usepackage{amsmath, amssymb, amsfonts,enumerate}
\usepackage[all]{xy}
\usepackage{amscd}
\usepackage{comment}
\usepackage{mathtools}
\usepackage{tikz} 
\usepackage{tikz-cd} 
\tikzcdset{diagrams={nodes={inner sep=7.5pt}}}

\usepackage{url}
\usepackage{hyperref}

\newtheorem{theorem}{Theorem}[section]
\newtheorem{lemma}[theorem]{Lemma}
\newtheorem{corollary}[theorem]{Corollary}
\newtheorem{proposition}[theorem]{Proposition}

 \theoremstyle{definition}
 \newtheorem{definition}[theorem]{Definition} 
  \newtheorem{conjecture}[theorem]{Conjecture}
 \newtheorem{remark}[theorem]{Remark}

 \newtheorem{example}[theorem]{Example}

\newtheorem{question}{Question}

\numberwithin{equation}{section}
 
\newcommand {\N}{\mathbb{N}} 
\newcommand {\Z}{\mathbb{Z}} 
\newcommand {\R}{\mathbb{R}} 
\newcommand {\Q}{\mathbb{Q}} 
\newcommand {\C}{\mathbb{C}} 
\newcommand{\F}{\mathbb{F}}
\newcommand{\G}{\mathbb{G}}

\newcommand{\FF}{\mathcal{F}}
\newcommand{\GG}{\mathcal{G}}

\newcommand{\PP}{\mathcal{P}}

\newcommand{\A}{\mathbb{A}}

\newcommand{\Mat}{\mathrm{Mat}}

\newcommand\sbullet[1][.5]{\mathbin{\vcenter{\hbox{\scalebox{#1}{$\bullet$}}}}}
\newcommand\numberthis{\addtocounter{equation}{1}\tag{\theequation}}
\newcommand*{\transp}[2][-3mu]{\ensuremath{\mskip1mu\prescript{\smash{\mathrm t\mkern#1}}{}{\mathstrut#2}}}%

\DeclareMathOperator{\mdim}{mdim}

\DeclareMathOperator{\CA}{CA}

\DeclareMathOperator{\Ker}{Ker}
\DeclareMathOperator{\End}{End}

\DeclareMathOperator{\M}{Mat}
\DeclareMathOperator{\im}{Im}

\DeclareMathOperator{\Id}{Id}

\DeclareMathOperator{\supp}{supp}
\DeclareMathOperator{\Spec}{Spec}
\DeclareMathOperator{\car}{char}

\begin{document}

\title[On sofic groups and endomorphisms of pro-algebraic groups]{On sofic groups, Kaplansky's conjectures, and endomorphisms of pro-algebraic groups}
\author[Xuan-Kien Phung]{Xuan Kien Phung}
\address{Universit\'e de Strasbourg, CNRS, IRMA UMR 7501, F-67000 Strasbourg, France}
\email{phung@math.unistra.fr}
\begin{abstract}
Let $G$ be a group. 
Let $X$ be a connected algebraic group over an algebraically closed field $K$. 
Denote by $A=X(K)$ the set of $K$-points of $X$. 
We study a class of endomorphisms of pro-algebraic groups, 
namely algebraic group cellular automata over $(G,X,K)$. 
They are cellular automata $\tau \colon A^G \to A^G$ whose local defining map is induced 
by a homomorphism of algebraic groups $X^M \to X$ 
where $M\subset G$ is a finite memory set of $\tau$. 
Our first result is that when $G$ is sofic, such an algebraic group cellular automaton $\tau$ is invertible 
whenever it is injective and $\car(K)=0$. 
When $G$ is amenable, 
we show that an algebraic group cellular automaton $\tau$ is surjective 
if and only if it satisfies a weak form of pre-injectivity called 
$(\sbullet[.9])$-pre-injectivity. This yields an analogue of 
the classical Moore-Myhill Garden of Eden theorem.  
We also introduce the near ring $R(K,G)$ which is $K[X_g: g \in G]$  
as an additive group but the multiplication is induced by the group law of $G$. 
The near ring $R(K,G)$ contains naturally the group ring $K[G]$ 
and we extend Kaplansky's conjectures to this new setting. 
Among other results, we prove that when $G$ is an orderable group,  
then all one-sided invertible elements  of $R(K,G)$ are trivial, 
i.e.,  of the form $aX_g+b$ for some $g\in G$, $a\in K^*$, and $b\in K$. 
This in turns allows us to show that when $G$ is locally residually finite 
and orderable (e.g. $\Z^d$ or a free group),  
all injective algebraic cellular automata $\tau \colon \C^G \to \C^G$ are of the form  
$\tau(x)(h)= a x(g^{-1}h) + b$ for all $x\in \C^G, h \in G$
for some  $g\in G$, $a\in \C^*$, and $b\in \C$. 
\end{abstract} 
\date{\today}
\maketitle

\setcounter{tocdepth}{1}
\tableofcontents

\section{Introduction}
For the notations, let $\N \coloneqq \{0,1,\dots\}$. 
We fix throughout an algebraically closed field $K$ unless stated otherwise. 
A $K$-algebraic variety means a reduced $K$-scheme of finite type 
and algebraic groups are assumed to be reduced. 
Every ring is assumed unital. 
\par
Cellular automata were first studied by von Neumann in the late 1940's as models
of computation and of artificial life. 
We first recall some basic definitions (see~\cite{ca-and-groups-springer}).  
\par
Fix a set $A$, called the \emph{alphabet},  and a group  $G$, called the \emph{universe}.
The set of maps $A^G \coloneqq  \{c \colon G \to A\}$ is called the set of \emph{configurations}. 
The Bernoulli shift action $G \times A^G \to A^G$ is defined by $(g,c) \mapsto g c$, 
where $(gc)(h) \coloneqq  c(g^{-1}h)$ for all $g,h \in G$ and $c \in A^G$. 
For $\Omega \subset G$ and $c \in A^G$, we define the \emph{restriction} 
$c\vert_\Omega \in A^\Omega$ by $c\vert_\Omega(g) \coloneqq  c(g)$ for all $g \in \Omega$. 
\par
A \emph{cellular automaton} over  the group $G$ and the alphabet $A$ is a map
$\tau \colon A^G \to A^G$ admitting a finite subset $M \subset G$
and a map $\mu \colon A^M \to A$ such that 
\begin{equation} 
\label{e:local-property}
(\tau(c))(g) = \mu((g^{-1} c )\vert_M)  \quad  \text{for all } c \in A^G \text{ and } g \in G.
\end{equation}
Such a set $M$ is called a \emph{memory set} of $\tau$ and $\mu$ is called the associated \emph{local defining map}. 
Every cellular automaton $\tau \colon A^G \to A^G$ is 
$G$-equivariant and continuous with respect to the prodiscrete topology on $A^G$.
 \par
Two configurations $c_1,c_2  \in A^G$ are \emph{almost equal} if 
$\{g \in G : c_1(g) \not= c_2(g) \}$ is a finite set.
 A cellular automaton $\tau \colon A^G \to A^G$ is \emph{pre-injective} if
 $\tau(c_1) = \tau(c_2)$ implies $c_1 = c_2$ whenever $c_1, c_2 \in A^G$ are almost equal. 
 A cellular automaton $\tau \colon A^G \to A^G$ is a \emph{linear cellular automaton} 
 if $A$ is a finite-dimensional vector space and $\tau$ is a linear map.
\par
In \cite{gromov-esav}, Gromov brought out a fascinating trialogue between group theory, 
symbolic dynamics, and algebraic geometry. 
The \emph{surjunctivity} property, i.e., injectivity implies surjectivity, 
is established in \cite{gromov-esav} for a certain class of endomorphisms of symbolic algebraic varieties 
under some variations on amenability hypotheses. 
For a finite alphabet $A$, Gromov and Weiss prove the surjunctivity of cellular automata 
$\tau \colon A^G \to A^G$ over any sofic group $G$ (cf.~\cite{gromov-esav},~\cite{weiss-sgds}). 
No non-sofic groups are currently known and the conjecture  
that this surjunctivity result  holds for any group $G$ is known as the
\emph{Gottschalk conjecture} (cf.~\cite{gottschalk}). 
Remarkably, for endomorphisms of algebraic varieties, the surjunctivity property is valid and known as the 
\emph{Ax-Grothendieck theorem} (cf.~\cite{ax-injective} and~\cite[Proposition~10.4.11]{ega-4-3})). 
\par
An interesting class of endomorphisms studied by Gromov in \cite{gromov-esav} is $CA_{alg}(G,X,K)$ 
consisting of \emph{algebraic cellular automata} over $(G,X,K)$. 
This class is introduced in a more general context in \cite{cscp-alg-ca} and \cite{cscp-alg-goe}. 
Here $G$ is a group and $X$ is a $K$-algebraic variety. 
More specifically, let $A \coloneqq X(K)$ denote the set of $K$-\emph{points} of $X$, i.e., 
$K$-scheme morphisms $\Spec(K) \to X$. 
A cellular automaton $\tau \colon A^G \to A^G$ is an algebraic cellular automaton
over $(G,X,K)$
if $\tau$ admits a memory set $M$ 
with local defining map $\mu \colon A^M \to A$ induced by composition with some $K$-scheme morphism 
$f \colon X^M \to X$, i.e., $\mu=f^{(K)}\colon A^M\to A$, 
where $X^M$ is the $K$-fibered product of copies of $X$ indexed by $M$
(cf.~\cite[Definition~1.1]{cscp-alg-ca}). 
When $G$ is trivial, we recover the notion of endomorphisms of an algebraic variety.  
\par
When $G$ is amenable and $K=\C$, 
the results in \cite{gromov-esav} imply that such algebraic cellular automata 
$\tau$ are surjunctive (cf.~\cite[2.C',~7.G']{gromov-esav}). 
It is shown in \cite[Theorem~1.2]{cscp-alg-ca}
that if $G$ is locally residually finite, and if $X$ is complete or $\text{char}(K)=0$, 
then injective algebraic cellular automata over $(G,X,K)$ are surjunctive. 
\par
A bijective cellular automaton is \emph{reversible} if its inverse is also a cellular automaton. 
When the alphabet is finite, it is well-known that every bijective cellular automaton is reversible 
(see e.g., \cite[Theorem 1.10.2]{ca-and-groups-springer}).
This property is false in general when the alphabet is infinite as 
the inverse map may have an infinite memory set (cf. \cite[Example 1.10.3]{ca-and-groups-springer}). 
However, \cite[Theorem~1.3]{cscp-alg-ca} settles the reversibility property 
in $CA_{alg}(G,X,K)$ for every group $G$ and every $K$-algebraic variety $X$ provided 
$K$ is uncountable or $X$ is complete. 
On the other hand, given $\tau \in CA_{alg}(G,X,K)$ which is reversible, 
it is not clear \emph{a priori} whether $\tau$ is \emph{invertible} in $CA_{alg}(G,X,K)$, namely,  
the inverse map $\tau^{-1}$ is also an element of $CA_{alg}(G,X,K)$. 
A positive answer is recently given in \cite[Theorem~1.4]{cscp-alg-ca} when $G$ is locally residually 
finite and $\car (K) = 0$. 
The result generalizes (by setting $G$ the trivial group) the invertibility property of endomorphism of 
algebraic varieties proved by Nowak (cf. \cite{nowak}). 
To this point, we can expect the surjunctivity and the invertibility property 
to hold in the class of algebraic cellular automata. More concretely, 

\begin{question}
\label{c:conjecture-alg}
Let $G$ be a group. Let $X$ be a $K$-algebraic variety. 
Suppose that $\car(K)=0$ and $\tau \in CA_{alg}(G,X,K)$ is injective. 
Is it true that $\tau$ is bijective and its inverse map $\tau^{-1} \colon A^G \to A^G$ is also an element of 
$CA_{alg}(G,X,K)$? 
\end{question}

Now let $G$ be a group. Let $X$ be a $K$-algebraic group and $A\coloneqq X(K)$. 
A cellular automaton $\tau \colon A^G \to A^G$ is  
an \emph{algebraic group cellular automaton over} $(G,X,K)$ 
if $\tau$ admits a memory set $M$ 
with local defining map $\mu \colon A^M \to A$ induced by some homomorphism of $K$-algebraic groups
$f \colon X^M \to X$. 
Denote $CA_{algr}(G,X,K) \subset CA_{alg}(G, X, K)$ the class consisting of 
algebraic group cellular automata over $(G,X,K)$. 
\par
The first goal of this paper is to settle Question \ref{c:conjecture-alg} 
for the class $CA_{algr}(G,X,K)$ whenever the group $G$ is sofic. 
More precisely, we obtain (cf. Section~\ref{s:sofic}): 

\begin{theorem}
\label{t:sofic}
Let $G$ be a sofic group. Let $X$ be a connected $K$-algebraic group. 
Suppose that $\car(K)=0$ and $\tau \in CA_{algr}(G,X,K)$ is injective. 
Then $\tau$ is surjective and hence bijective. 
Moreover, its inverse map $\tau^{-1} \colon A^G \to A^G$ is also an element of $CA_{algr}(G,X,K)$.  
\end{theorem}

Remark that Theorem \ref{t:sofic} is false in positive characteristic 
because of the Frobenius map (cf. Example \ref{ex:main-frobenius}).  
The condition $\car (K)=0$ is thus necessary in both Theorem \ref{t:sofic} 
and in Question \ref{c:conjecture-alg}. 
Every finite dimensional vector space $A$ over the algebraically closed field $K$ is a $K$-algebraic group. 
Hence, Theorem \ref{t:sofic} generalizes the corresponding known result for the class $CA_{lin}(G, A, K)$ 
consisting of linear cellular automata  $\tau \colon A^G \to A^G$ (cf. \cite{csc-sofic-linear}). 
When $G$ is residually finite, the connectedness condition in Theorem~\ref{t:sofic} on $X$ 
can be dropped as follows (cf. Theorem~\ref{t:surj-grp} and Section~\ref{s:reverse-inverse}). 

\begin{theorem}
\label{t:surjunctive-res}
Let $G$ be a locally residually finite group. 
Let $X$ be a $K$-algebraic group. 
Suppose that $\car(K)=0$ and $\tau \in CA_{algr}(G,X,K)$ is injective.  
Then $\tau$ is bijective. 
Moreover, the inverse map $\tau^{-1}\colon A^G \to A^G$ is also an element of $CA_{algr}(G,X,K)$.  
\end{theorem}

Remark that the algebraic group $X$ is not necessarily complete 
in Theorem \ref{t:surjunctive-res} and Theorem \ref{t:sofic}. 
Hence, our result extends the result of \cite[Theorem~1.2]{cscp-alg-ca}. 
\par
A ring $T$ is \emph{directly finite} or \emph{Dedekind finite} if for all $c,d \in T$, $cd=1$ implies $dc=1$. 
The ring $T$ is \emph{stably finite} if every matrix ring $\Mat_n(T)$ ($n \geq 1$) is directly finite. 
Kaplansky conjectured in \cite{kap} that for every field $k$ and any group $G$, 
the group ring $k[G]$ (i.e., 
the $k$-algebra with $G$ as basis and multiplication defined by the group
product on basis elements and the distributive law) is directly finite. 
He proved the conjecture himself in \cite{kap} when $k=\C$. 
With the theory of operator algebras, it is shown in \cite{elek-szabo} 
that $k[G]$ is directly finite for every field $k$ and every sofic group $G$. 
In fact, $k[G]$ is directly finite for every sofic group $G$ whenever $k$ is a division ring 
(cf. \cite{elek-szabo} and \cite{ara}) or more generally a left (or right) Artinian ring (cf. \cite{csc-artinian}). 
It is natural to expect that the group ring $R[G]$ is directly finite 
for every directly finite ring $R$ and every group $G$. 
Every left Noetherian ring is directly finite and it is established recently in \cite{li-liang} that 
$R[G]$ is stably finite for every left Noetherian ring $R$ and every sofic group $G$ with the tool of 
sofic mean length introduced in the same paper \emph{loc. cit.}. 
To the author's knowledge, results in \cite{ara}, \cite{csc-artinian}, \cite{elek-szabo}, and \cite{li-liang} 
are the most general on the stably direct finiteness of $R[G]$. 
\par
Using the ring isomorphism $k[G]\simeq CA_{lin}(G, k, k)$, another proof 
of the direct finiteness of $k[G]$ for $k$ a field and $G$ a sofic group  
is given in \cite{csc-sofic-linear} by establishing the surjectivity property for elements of $CA_{lin}(G, k, k)$. 
In general, for a group $G$ and a $K$-algebraic variety $X$,  $CA_{alg}(G,X,K)$ is a monoid 
for the composition of maps (cf. \cite[Proposition 3.3]{cscp-alg-ca}) 
and we have the following logical relation for elements of $CA_{alg}(G,X,K)$:  
\begin{equation}
\label{logical-direct-finite-surjunctive}
\text{surjunctivity}  \implies \text{direct finiteness}
\end{equation}
as follows. Assume the surjunctivity for $CA_{alg}(G,X,K)$ and let 
$\sigma, \tau \in CA_{alg}(G,X,K)$ with $\sigma \circ \tau = \Id$. 
Then $\tau$ is clearly injective. 
By surjunctivity, $\tau$ is thus bijective and it follows at once that $\tau \circ \sigma= \Id$. 
Hence, we can use pointwise method and properties of cellular automata such as the surjunctivty  
to obtain formal point-free results such as the stable finiteness 
for the group ring $K[G] \simeq CA_{lin}(G, K, K)$. 
The reverse procedure is also very helpful. Indeed, let $\A^1$ denote the affine line. 
We shall see below (cf. the discussion after Question \ref{injective-map-ca}) 
that in each case where Question \ref{c:conjecture-alg} holds true, 
the unit conjecture (cf. Conjecture \ref{c:conjecture}, Theorem \ref{t:unit-kap}) for the near ring $R(K, G) \simeq CA_{alg}(G, \A^1, K)$ 
(cf. \eqref{e:iso-near-ring-ca-alg}) implies 
a classification of injective maps in $CA_{alg}(G, \A^1, K)$ (cf. Question \ref{injective-map-ca}, Theorem \ref{t:classification}). 
\par
There is a natural $R$-linear ring isomorphism $R[G] \simeq CA_{algr}(G,X,K)$ 
where $R=\End(X)$ is the endomorphism ring of a commutative algebraic group $X$ 
(cf. Proposition \ref{p:iso-ca-gr}). 
As an application, our results on $CA_{algr}(G,X,K)$ 
implies following stable finiteness of group rings (cf. Corollary \ref{c:kap-direct-finite}). 
 
\begin{corollary}
\label{t:kap-direct-finite}
Given a sofic group $G$ and a commutative $K$-algebraic group $X$ with endomorphism ring $R=\End(X)$, 
the group ring $R[G]$ is stably finite in the following cases:  
\begin{enumerate} [\rm (i)]
\item
$G$ is sofic and $X$ is connected;
\item
$G$ is locally residually finite. 
\end{enumerate}
\end{corollary}

By setting $X=K^n$, $R= \Mat_n(K)$, one recovers the results 
in \cite{elek-szabo} and \cite{csc-sofic-linear}. 
To the limit of the author's knowledge, it seems, unfortunately,  
that there is no examples of endomorphism rings of a commutative algebraic group 
that cannot be embedded inside a Noetherian ring. 
\par
With a group $G$ and a ring $L$, one can associate a near ring $(R(L,G), +, \star)$  
which is $(L[X_g: g \in G], +)$ as an additive group 
but its multiplication law $\star$ is induced by the group law of $G$ 
as a generalized form of the convolution product (cf. Section \ref{s:near-ring}). 
$R(L,G)$ contains the group ring $L[G]$ as the subring 
consisting of homogeneous polynomials of degree $1$ (Proposition \ref{p:embed-ring}). 
We can thus extend accordingly Kaplansky's conjectures to $R(L,G)$ (cf. Conjecture \ref{c:conjecture}). 
The near ring $R(L,G)$ arises naturally in the context of algebraic cellular automata 
via a canonical isomorphism of near rings (Theorem \ref{t:iso-nr-ca}) 
\begin{equation}
\label{e:iso-near-ring-ca-alg}
\Phi \colon R(L,G)\to CA_{alg}(G, \A^1, L)
\end{equation}
whenever $L$ is commutative and such that the polynomial ring $L[T]$ 
has no nonzero polynomial that vanishes identically as a function $L \to L$ (e.g., $L$ is an infinite field). 
\par
Mal'cev's results on Kaplansky's unit conjecture for orderable groups \cite{malcev} 
can be extended to the near ring $R(L,G)$ (cf. Theorem \ref{t:kap-unit}) as follows: 

\begin{theorem}
\label{t:unit-kap}
Given an orderable group $G$ and an integral domain $L$, 
if  $\alpha, \beta \in R(L, G)$ and $\alpha \star \beta =X_{1_G}$ then $\alpha=aX_g-ab$, 
$\beta= a^{-1}X_{g^{-1}}+b$ for some $g\in G$, $a\in L^*$, and $b\in L$. 
\end{theorem}

From Theorem \ref{t:unit-kap} and results in \cite{cscp-alg-ca}, 
we obtain a nontrivial classification 
of injective maps in $CA_{alg}(G,\A^1,\C)$ (cf. Corollary \ref{c:classification})). 
The result can be seen as a generalization of the trivial fact saying that 
every polynomial injective self-map $\tau \colon \C \to \C$ must be of the form 
$\tau(x)= ax+b$ for some $a\in \C^*$ and $b \in \C$. 

\begin{theorem}
\label{t:classification} 
Let $G$ be an orderable and locally residually finite group (e.g. $\Z^d$ or a free group). 
Assume $\tau \in CA_{alg}(G, \A^1, \C)$ is injective. 
Then there exists $g\in G$, $a \in \C^*$, and $b\in \C$
such that for all $x \in \C^G$ and $h \in G$, we have 
\[
\tau(x)(h)= a x(g^{-1}h) +b. 
\]
\end{theorem}

Therefore, Theorem \ref{t:classification} gives some evidence for the following classification question: 

\begin{question}
\label{injective-map-ca}
Let $G$ be a torsion free group. 
Suppose that $\tau \in CA_{alg}(G, \A^1, \C)$ is injective. 
Does there exist $g\in G$, $a \in \C^*$, and $b\in \C$
such that for all $x \in \C^G$ and $h \in G$, we have 
\[
\tau(x)(h)= a x(g^{-1}h) +b\,? 
\]
\end{question}

Our main supporting motivation for Question \ref{injective-map-ca} is that it is 
a direct consequence of Question \ref{c:conjecture-alg} 
and the unit Kaplansky's conjecture \ref{c:conjecture} for $R(\C, G)$. 
Indeed, let $\tau \in CA_{alg}(G, \A^1, \C)$ and let $\alpha = \Phi^{-1} (\tau) \in R(\C, G)$ (cf. \eqref{e:iso-near-ring-ca-alg}). 
If $\tau$ is injective, then Question \ref{c:conjecture-alg} implies that $\tau$ is invertible in $CA_{alg}(G, \A^1, \C)$. 
Hence, $\alpha$ is also invertible in $R(\C, G)$ and by the unit Kaplansky's conjecture \ref{c:conjecture}, 
$\alpha=aX_g+b$ for some $g\in G$, $a\in \C^*$, and $b\in \C$. 
This in turns implies that $\tau(x)(h)= a x(g^{-1}h) +b$  for all $x \in \C^G$ and $h \in G$. 
Remark that the condition saying $G$ is torsion free is necessary in both  
the unit Kaplansky's conjecture \ref{c:conjecture} for $R(\C, G)$ and Question \ref{injective-map-ca} 
(cf. Example \ref{ex:torsion-free-injective}). 
\par
As $R(L,G)$ is not a ring in general, 
it is not clear \emph{a priori} if, for $R(L,G)$, the unit conjecture implies 
the zero-divisor and the idempotent conjectures, as for the group ring $L[G]$. 
However, we can obtain the following properties on solutions of polynomial equations in $R(L,G)$ (cf. Section \ref{s:near-ring}). 
For $\alpha \in R(L,G)$ and a polynomial $P(x)=\sum_n c_n x^n \in L[x]$, 
define $P(\alpha)\coloneqq c_0 + \sum_{n\geq 1} c_n \alpha^{(\star n)} $ 
where $\alpha^{(\star n)}\coloneqq \alpha \star \dots \star \alpha \text{ (n-times)}$ for $n\geq 1$. 

\begin{theorem}
\label{t:zero-kap}
Let $G$ be an orderable group and let $L$ be a domain. 
For all $\alpha, \beta \in R(L, G)$, $a\in L$ and $P\in L[x]$ a nonconstant polynomial, we have: 
\begin{enumerate} [\rm (i)]
\item
$\alpha \star \beta = 0$ implies $\alpha=0$ or $\beta$ is a constant;
\item
$P(\alpha) = aX_{1_G}$ implies $\alpha = cX_{1_G}+d$ for some $c,d \in L$.  
In particular, $\alpha^{(\star 2)}=\alpha$ implies $\alpha=X_{1_G}+d$ where $d\in L$, $2d=0$ or $\alpha$ is a constant. 
\end{enumerate}
\end{theorem}

Remark again that Theorem 
\ref{t:zero-kap} and Isomorphism \eqref{e:iso-near-ring-ca-alg} imply, for example, 
that if the group $G$ is orderable, every idempotent algebraic cellular automaton 
$\tau \colon \C^G \to \C^G$ must be the identity or a constant map: 
$\tau (x) = x$ for every $x \in \C^G$ or $\tau(x) \equiv c^G$ for some $c \in \C$. 
\par
The classical Myhill-Moore Garden of Eden theorem (cf.~\cite{myhill}, \cite{moore}) states that 
a cellular automaton with finite alphabet over the group $\Z^d$ is pre-injective if and only if it is surjective. 
It is an equivalence between a \emph{local property} (pre-injectivity) and a \emph{global} one (surjectivity). 
As injectivity $\implies$ pre-injectivity, Myhill-Moore Garden of Eden theorem strengthens the surjunctivity property 
for cellular automata over a finite alphabet when $G= \Z^d$.  
The theorem was extended to cellular automata with finite alphabet over amenable groups in~\cite{ceccherini}.
Another version in~\cite{csc-garden-linear} states that 
a linear cellular automaton is pre-injective if and only if it is surjective, 
which gives a characterization of amenable groups (cf.~\cite{bartholdi}, \cite{bartholdi-kielak}). 
See the survey \cite{goe-survey} for the history and various recent research related to Garden of Eden theorem. 
As for the class $CA_{alg}$, Gromov asked at the end of \cite{gromov-esav}:

\begin{quote}
8.J. Question. Does the Garden of Eden theorem generalize to the 
proalgebraic category?
First, one asks if pre-injective  $\implies$  surjective, while the reverse implication 
needs further modification of definitions.
\end{quote}

\par
Results in \cite{cscp-alg-goe} give a positive answer to Gromov's question 
for algebraic cellular automata over $(G,X,K)$ where $G$ is an amenable group 
and $X$ is an irreducible complete $K$-algebraic variety. 
The authors show that the \emph{Myhill property}, i.e., pre-injectivity implies surjectivity, 
holds for such algebraic cellular automata,  while the \emph{Moore property}, 
i.e., surjectivity implies pre-injectivity, is false in general because of non bijective covering maps, 
but it holds if we replace pre-injectivity by a weaker notion, namely $(*)$-pre-injectivity 
(cf. \cite[Theorem~1.1, Theorem~1.3]{cscp-alg-goe}). 
Moreover, it was shown in  \cite{cscp-alg-goe} that surjectivity and $(*)$-pre-injectivity are equivalent. 
This yields an analogue of the Garden of Eden theorem for (complete) algebraic cellular automata. 
\par
It is not known whether the Myhill property holds for all algebraic cellular automata over amenable groups.  
Note that there are counter-examples when $G$ is not an amenable group 
(see~\cite[Example~8.5]{cscp-alg-goe}). 
To establish similar results for algebraic group cellular automata, 
we introduce a refined notion of $(*)$-pre-injectivity called $(\sbullet[.9])$-pre-injectivity 
(cf.~Definition~\ref{d:weak-pre}) which is also implied by pre-injectivity. 
Using the algebraic mean dimension $ \mdim_\FF$ (cf. Section \ref{s:alg-mean-dim}), 
we shall prove the following (cf. Theorem~\ref{t:main-goe}): 

\begin{theorem}
\label{t:goe-grp}
Let $G$ be an amenable group and let $X$ be a connected $K$-algebraic group. 
Let  $\tau \in CA_{algr}(G,X,K)$. 
Then $\tau$ is surjective if and only if $\tau$ is $(\sbullet[.9])$-pre-injective. 
\end{theorem}

\begin{corollary} 
\label{c:myhill}
Let the hypotheses be as in Theorem \ref{t:goe-grp}, then $\tau$ has the Myhill property, 
i.e., $\tau$ is surjective whenever it is pre-injective. 
\end{corollary}
\par
Our paper is organized as follows. 
Background materials are collected in Section \ref{s:pre}.
We recall some fundamental theorems of algebraic groups and necessary properties of sofic and amenable groups. 
The definition and elementary properties of algebraic group cellular automata are presented in Section \ref{s:alg-grp-ca}. 
We recall also in this section the notion of algebraic mean dimension 
$ \mdim_\FF$ introduced in \cite{cscp-alg-goe}. 
In Section~\ref{s:criterion-inverse}, we give a nonemptyness criterion 
for the limit of an inverse system of algebraic varieties (Lemma~\ref{l:C}). 
A fundamental property of images of cellular automata, namely 
the closed image property (with respect to the prodiscrete topology), 
is established for algebraic group cellular automata in Section~\ref{s:cip-grp} (Theorem~\ref{t:closed-image}). 
As an application, we obtain the surjunctivity of algebraic group cellular automata 
when the universe is a locally residually finite group (Theorem~\ref{t:surj-grp}). 
Section~\ref{s:reverse-inverse} contains the key technical result (Proposition~\ref{p:left-inverse}) 
necessary for the proof of Theorem~\ref{t:sofic} presented in Section~\ref{s:sofic}. 
Proposition~\ref{p:left-inverse} yields also the invertibility 
of any bijective algebraic group cellular automata (Theorem~\ref{t:inverse-also}) 
and completes the proof of Theorem~\ref{t:surjunctive-res}. 
We introduce analogous notions of $(*)$- and $(**)$-pre-injectivity given in \cite{cscp-alg-goe}, 
namely $(\sbullet[.9])$- and $( \sbullet[.9] \sbullet[.9]) $-pre-injectivity (cf.~Definition~\ref{d:weak-pre}) 
and shows that they are all equivalent in the context of Theorem~\ref{t:goe-grp} (cf. Proposition~\ref{p:pre-injectivity-*}). 
Using results in \cite{cscp-alg-goe} and Theorem~\ref{t:closed-image}, a variant of the Garden of Eden theorem 
for algebraic group cellular automata (Theorem~\ref{t:main-goe}) is proved 
which implies also Theorem~\ref{t:goe-grp} and Corollary~\ref{c:myhill}. 
The proof of Theorem \ref{t:kap-direct-finite} is presented in Section \ref{s:kap-finite}. 
The definition of the near ring $R(K,G)$ and the proof of Theorem \ref{t:unit-kap}, 
Theorem \ref{t:zero-kap} are discussed in Section \ref{s:near-ring}. 
Several counter-examples are given in Section~\ref{s:counter}. 
\par
Suppose that $G$ is an amenable group and $X$ is a $K$-algebraic group. 
Let $A\coloneqq X(K)$ and $\tau  \in CA_{algr}(G,X,K)$. 
Our results concerning the Garden of Eden theorem for $\tau$ can be summarized in the following diagram. 
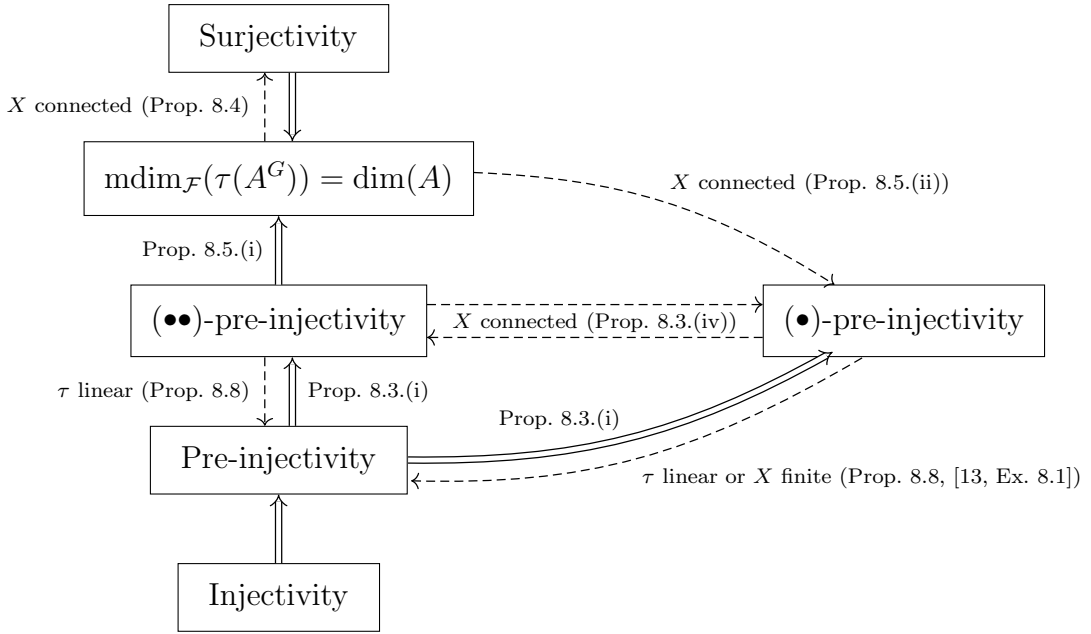
\[
\begin{tikzcd}
[cells={nodes={draw=black}}, column sep=4.6em, row sep=2.2em]
 & \text{ Surjectivity } 
 \arrow[d, Rightarrow, shift left=1ex]  
 & &   \\
 & \mdim_\FF(\tau(A^G))=\dim(A) 
 \arrow[u, "X \text{ connected (Prop.~\ref{p:sur-dim}) }", dashrightarrow, shift left=1ex] 
 \arrow[d, "\text{ Prop.~\ref{p:**-implies-mdim}.(i) }", Leftarrow, swap] 
 \arrow[drr, "X \text{ connected (Prop.~\ref{p:**-implies-mdim}.(ii))}", dashrightarrow, bend left=15] 
&  &\\
& ( \sbullet[.9] \sbullet[.9]) \text{-pre-injectivity} 
\arrow[rr,"X\text{ connected (Prop.~\ref{p:pre-injectivity-*}.(iv))}", dashrightarrow, shift left=1.2ex, swap] 
\arrow[d,  "\tau \text{ linear (Prop.~\ref{p:preinj-lin}) }" , dashrightarrow, shift right=1ex,swap] 
& & (\sbullet[0.9])\text{-pre-injectivity} 
 \arrow[ll,  dashrightarrow, shift left=1.2ex]  
 \arrow[dll, "\tau \text{ linear or } X \text{ finite (Prop.~\ref{p:preinj-lin}, \cite[Ex.~8.1]{cscp-alg-goe})}", dashrightarrow, bend left=16, shorten <= -0.8em, shift left=1.ex] \\
 & \text{  Pre-injectivity  }  
  \arrow[u, "\text{ Prop.~\ref{p:pre-injectivity-*}.(i) }", Rightarrow, shift right=1ex, swap]
  \arrow[urr, "\text{ Prop.~\ref{p:pre-injectivity-*}.(i)}", Rightarrow, shift left=0.4ex, shorten <= 0.06em, bend right=15] &  & \\
  & \text{ Injectivity } 
  \arrow[u, Rightarrow]
  & & 
\end{tikzcd}
\]

\section{Notations and preliminary results} 
\label{s:pre}
\subsection{Algebraic varieties}
For each $K$-algebraic variety $X$, 
we shall frequently use the canonical identification $X(K)=X_0$, 
where $X_0$ denotes the set of closed points of $X$ (cf. \cite[Corollaire~6.4.2]{grothendieck-ega-1}). 
We equip $A \coloneqq X(K)$ with the induced Zariski topology on $X_0$. 
Then $A$ is a Noetherian $T_1$ topological space. 
Every point of $A$ is a closed point and every subspace is quasi-compact.  
Note that $A$ is nonempty whenever $X$ is nonempty. 
Let $f\colon X \to Y$ be a morphism of $K$-algebraic varieties. 
Then $f$ induces by composition a map on $K$-points $f^{(K)} \colon X(K) \to Y(K)$ 
which coincides  with the map $f\vert_{X_0} \colon X_0 \to Y_0$
obtained by restricting $f$ to closed points (see, e.g., \cite[Lemma~A.22]{cscp-alg-ca}).
When there is no possible confusion, we shall simply write $f$ for both $f\vert_{X_0}$ and $f^{(K)}$.  
\par
For each nonempty subset $C\subset X$, we denote by $\dim(C)\in \N$ 
the Krull dimension of $C$ as a subspace of $X$ (cf.~\cite[2.5.1]{liu-alg-geom}). 
Note that $\dim(A)=\dim(X)$ and $\dim(C)=\dim(C_0)$ for every constructible subset 
$C\subset X$ (see, e.g., \cite[Proposition~2.9]{cscp-alg-goe}).   
Each subset $F \subset A$, viewed as a subspace of $A$ or of $X$, has a Krull dimension denoted by $\dim(F)$.  
\par
Note that if $X$ is an irreducible $K$-algebraic variety and $E$ is a finite set, 
then the fibered product $X^E$ is also irreducible (cf.~\cite[Proposition~5.50]{gorz}).

\subsection{Algebraic groups} 
Let $k$ be a (possibly not algebraically closed) field. 
A \emph{$k$-algebraic group} is a quadruple $(X,m,\text{inv},e)$, where
$X$ is a $k$-scheme of finite type,  
$m \colon X\times_k X \to X$ and $\text{inv}\colon X \to X$ 
are morphisms of $k$-schemes, and $e \in X(k)$ verify formally the axioms of a group (cf. \cite[Defintion~1.1]{milne})
The point $e$ is called the \emph{neutral} or \emph{identity} element. 
We shall sometimes denote an algebraic group $(X,m,\text{inv},e)$
by $(X,m)$ or $(X,e)$ or simply $X$ if the algebraic group structure is clear from the context.  
Since we are interested only in the set of $k$-points, we make the convention that algebraic groups are reduced. 
\par
To simplify notation, we shall often use the symbol $+\coloneqq m$ for the group law of algebraic groups. 
We also use the notation  $-\coloneqq +\circ (\Id, \text{inv})$. 
A $k$-scheme morphism $\varphi \colon X \to Y$ of algebraic groups is called 
a \emph{homomorphism of algebraic groups} if $\varphi \circ m_X = m_Y \circ (\varphi \times \varphi)$. 
Here $m_X$, $m_Y$ are respectively the group laws of $X$ and $Y$. 
Typical examples of algebraic groups are \emph{linear algebraic groups} and \emph{abelian varieties}. 

\begin{remark}
\label{r:product-group}
Let $(X,m_X, \text{inv}_X,e_X)$ and $(Y,m_Y,\text{inv}_Y,e_Y)$ be $k$-algebraic groups. Then: 
\begin{enumerate}[\rm(1)]
\item
$Z=X\times_k Y$ is a $k$-algebraic group   
with neutral element $e=(e_X, e_Y)$ and inverse morphism $\text{inv}_X \times \text{inv}_Y$. 
The group law is given by: 
\[
m\colon Z\times_k Z \xrightarrow{\simeq} (X_{(1)}\times_k X_{(2)}) \times_k (Y_{(1)}\times_k Y_{(2)}) 
 \xrightarrow{m_X\times m_Y} X \times_k Y=Z
\]  
where $Z=X_{(i)}\times_k Y_{(i)}$ is the $i$-th factor in the fibered product $Z\times_k Z$ for $i=1,2$. 
\item
Let $N\subset M$ be two nonempty finite sets. 
Then the projection $p\colon X^M \to X^N$ is a homomorphism of $k$-algebraic groups. 
\item
Let $f \colon T \to X$ and $g \colon T \to Y$ be homomorphisms of $k$-algebraic groups.  
Then $h\coloneqq (f,g)  \colon T \to X \times_k Y$ is also a homomorphism of $k$-algebraic groups. 
\end{enumerate}
\end{remark}

\begin{remark}
Every linear cellular automaton $\tau \colon A^G \to A^G$ over a group $G$ and a 
vector space $A\simeq K^n$ ($n\in \N$) is naturally an algebraic group cellular automaton 
over $(G,\G_a^n, K)$ where $\G_a^n\simeq \A^n$ is the $n$-dimensional vector group. 
\end{remark}

\begin{theorem}[Homomorphism theorem]
\label{t:homo-group} 
Let $f \colon X\to Y$ be a homomorphism of $K$-algebraic groups. 
Then $\Ker(f)$ is a normal closed algebraic subgroup of $X$. 
Moreover, $\im(f)$ is a closed algebraic subgroup of $Y$ and we have 
the following commutative diagram of homomorphisms: 
\[
\begin{tikzcd}
X \arrow[d,"\pi"]   \arrow[r,"f"] & Y \\
X/ \Ker(f) \arrow[r, " \simeq"] & \im(f) \arrow[u, hook, swap]
\end{tikzcd}
\]
where $\pi \colon X \to X/ \Ker(f)$ is a quotient map of $K$-algebraic groups and 
$\im(f) \hookrightarrow Y$ is a closed immersion homomorphism. 
\end{theorem}

\begin{proof}
See \cite[Theorems~5.80]{milne} and \cite[Theorem~ 5.81]{milne}. 
Note that by \cite[Proposition~1.41]{milne}, every algebraic subgroup of an algebraic group 
is closed in the Zariski topology.  
\end{proof}

A homomorphism of algebraic groups is called a \emph{quotient map} 
if it is faithfully flat, or equivalently, surjective and flat. 
Every surjective homomorphism $X\to Y$ of algebraic groups, 
where $Y$ is reduced, is a quotient map. 

\begin{theorem}
\label{t:image}
Every homomorphism of algebraic groups $\varphi \colon X \to Y$ factors 
(essentially uniquely) into a composite of homomorphisms 
\[
X \xrightarrow{q} Z \xrightarrow{\iota} Y
\]
where $q$ is a quotient map and $\iota$ is an embedding, i.e., a closed immersion homomorphism. 
\end{theorem}

\begin{theorem}[Universal property of quotient maps]
\label{t:quotient-universal}
Let $\pi \colon X \to Y$ be a quotient map of $K$-algebraic groups 
with kernel $N=\Ker(\pi)$. 
Then every homomorphism $\varphi \colon X \to Z$ of $K$-algebraic groups 
whose kernel contains
$N$ factors uniquely through $\pi$ by a homomorphism: 
\[
\begin{tikzcd}
X \arrow[dr,"\pi",swap]   \arrow[r,"\varphi"] & Z \\
&  Y \arrow[u,"\exists ! h",dashrightarrow, swap]
\end{tikzcd}
\]
\end{theorem}

\begin{proof}
See \cite[Theorem~5.13]{milne}.  
\end{proof}

Note also that by~\cite[Corollary~1.35]{milne}, a $K$-algebraic group is irreducible 
if and only if it is connected. In characteristic zero, Cartier's theorem (cf. \cite[Corollary~11.31]{milne}) 
implies that every algebraic group is reduced and hence smooth. 
 
\subsection{Amenable and residually finite groups} 
\label{ss:sofic-groups}
A group $G$ is called \emph{amenable} if there exist a directed set $I$ and a family $(F_i)_{i \in I}$ 
of nonempty finite subsets of $G$  such that
\begin{equation}
\label{e:folner-s}
\lim_{i} \frac{\vert F_i \setminus  F_i g\vert}{\vert F_i \vert} = 0 \text{  for all } g \in G. 
\end{equation}
Such a family $(F_i)_{i \in I}$ is called a \emph{(right) F\o lner net} for $G$. 
\par
Finitely generated groups of subexponential growth and solvable groups are amenable 
while groups containing a non-abelian free subgroup are non-amenable.  
\par
A group $G$ is called \emph{residually finite} if the intersection of all finite-index
subgroups of $G$ is reduced to the identity element. Every finitely generated abelian group
and, more generally, every finitely generated
linear group is residually finite (Mal'cev's theorem). 
\par
Let $T$ be a property of groups. A group $G$ is called \emph{locally $T$} 
if every finitely generated subgroup of $G$ verifies $T$.  

\subsection{Geometry of sofic groups}
Sofic groups were first introduced by Gromov \cite{gromov-esav} and Weiss \cite{weiss-sgds}. 
This class of groups contains all residually finite groups and all amenable groups. 
It is not known up to now whether all groups are sofic or not. 
Moreover, several conjectures for groups have been established for the sofic ones such as 
Gottschalk's surjunctivity conjecture or the Kaplansky's direct finiteness conjecture. 
For more details and history, readers may also consult \cite{capraro}, \cite{ca-and-groups-springer}. 
 \par
Every subgroup of a sofic group is also sofic (cf.~\cite[Proposition~7.5.4]{ca-and-groups-springer}). 
Moreover, a group is sofic if and only if it is locally sofic (cf.~\cite[Proposition~7.5.5]{ca-and-groups-springer}). 
For finitely generated sofic groups, we shall rather use in this paper the characterization by 
the local approximation of their Cayley graphs by finite labelled graphs (Theorem~\ref{t:sofic-character}). 
\par
Let $S$ be a finite set. We define an \emph{$S$-label graph} as a pair $\GG= (V,E)$, 
where $V$ is a set called the set of \emph{vertices}, 
and $E$ is a subset of $V \times S \times V$, called the set of \emph{edges}. 
\par
The associated \emph{labeling map} is the projection $\lambda \colon E \to S$ given by 
$\lambda(c)=s$ for all $c=(v,s,v')\in E$. $\GG$ is said to be \emph{finite} if $V$ is finite. 
An $S$\emph{-labeled subgraph} of $\GG$ is an $S$-labeled graph $(V',E')$ 
such that $V' \subset V$ and $E' \subset E$.
\par
Denote by $l(\rho)$ 
the \emph{length} of a path $\rho$ in $\GG$. 
If $v, v'\in V$ are not connected by a path in $\GG$, we set $d_\GG(v,v')=\infty$. 
Otherwise, we define  
\[
d_\GG(v,v')\coloneqq \min \{ l(\rho): \text{$\rho$ is a path from $v$ to $v'$}\}. 
\]
For $v\in V$ and $r \in \N$, the \emph{ball $B_\GG(v,r) \subset V$ centered at $v$ of radius $r$} is defined by 
\[
B_\GG(v,r)\coloneqq \{v'\in V: d_\GG (v,v') \leq r\}. 
\]
Note that $B_\GG(v,r)$ is finite and inherits a natural structure of $S$-labeled subgraph of $\GG$. 
\par
Let $(V_1,E_1)$ and $(V_2,E_2)$ be two $S$-label graphs. 
A map $\phi \colon V_1 \to V_2$ is called an \emph{$S$-labeled graph homomorphism} 
from $(V_1, E_1)$ to $(V_2,E_2)$ if $(\phi(v),s, \phi(v')) \in E_2$ for all $(v,s,v')\in E_1$. 
\par
A bijective $S$-labeled graph homomorphism $\phi \colon V_1 \to V_2$ is an 
\emph{$S$-labeled graph isomorphism} if the inverse map $\phi^{-1}\colon V_2 \to V_1$ 
is also an  $S$-labeled graph homomorphism. 
\par
Now let $G$ be a finitely generated group and let $S\subset G$ 
be a finite \emph{symmetric} generating subset, i.e., $S=S^{-1}$. 
The \emph{Cayley graph of $G$ with respect to $S$} is the connected $S$-labeled graph $C_S(G) = (V,E)$,  
where $V = G$ and $E=\{(g,s,gs): g\in G \text{ and } s \in S)\}$.  
We have the following compatibility result (cf.~\cite[Corollary~6.2.2]{ca-and-groups-springer} 
and \cite[Proposition~6.3.1]{ca-and-groups-springer})

\begin{proposition}
Let $G$ be a finitely generated group and let $S \subset G$ be a finite symmetric generating subset of $G$. 
The map $d_S \colon G\times G \to \N$ defined by:  
\[
d_S(g,h)\coloneqq \min\{n \geq 0: g^{-1}h=s_1s_2\dots s_n,\text{ where } s_i \in S \,, \, 1 \leq i \leq n\}
\]
is a well-defined metric on $G$. 
Moreover, for all $g,h \in G$, we have   
\begin{equation}
\label{e:metric}
d_S(g,h) = d_{C_S(G)}(g,h). 
\end{equation}
\end{proposition}

For $g \in G$ and $r\in \N$, the \emph{ball in $G$ centered at $g$ of radius $r$} is the finite subset 
\[
B_S(g,r)\coloneqq \{h \in G:  d_S(g,h)\leq r \}. 
\]  
When $g=1_G$, denote $B_S(r)\coloneqq B_S(1_G,r)$. 
Finitely generated sofic groups are characterized in terms of a finiteness condition 
on their Cayley graphs as follows (cf. \cite[Theorem~7.7.1]{ca-and-groups-springer}).

\begin{theorem}
\label{t:sofic-character}
Let $G$ be a finitely generated group. 
Let $S\subset G$ be a finite symmetric generating subset. 
Then the following are equivalent: 
\begin{enumerate} [\rm (a)]
\item
the group $G$ is sofic;
\item
for all $r\in \N$ and $\varepsilon >0$, there exists a finite $S$-labeled graph $\GG=(V,E) $ 
satisfying 
\begin{equation} 
\label{e:sofic-Q(3r)}
\vert V(r) \vert \geq (1 - \varepsilon) \vert V \vert,
\end{equation}
where $V(r)\subset V$ consists of $v \in V$ such that there exists a (unique) $S$-labeled graph
isomorphism $\psi_{v,r} \colon B_S(r) \to B_\GG(v,r)$ with $\psi_{v,r}(1_G) = v$. 
\end{enumerate}
\end{theorem}

Since each $S$-labeled graph isomorphism $\psi_{v,r'} \colon B_S(r') \to B_\GG(v,r')$ induces by restriction  
an $S$-labeled graph isomorphism $B_S(r) \to B_\GG(v,r)$ for all $r'\geq r$, we clearly have: 
\[
V(r) \supset V(2r) \supset V(3r) \supset \cdots. 
\]
We shall also need the following auxiliary well-known lemma 

\begin{lemma}
\label{l:sofic-B-V} \label{l:sofic-V'}
With the notation as in Theorem~\ref{t:sofic-character}, the following hold 
\begin{enumerate}[\rm(i)]
\item
$B_\GG(v,r) \subset V(kr)$ for all $v \in V((k+1)r)$ and $k \geq 0$;
\item
there exists a finite subset $V' \subset V(3r)$ such that the balls $B_\GG(q,r)$ 
are pairwise disjoint for all $v\in V'$ and that 
$V(3r) \subset \bigcup_{v\in V'} B_\GG(v,2r)$. 
In particular, we have 
\begin{equation*}
\vert B_S(2r) \vert \vert V' \vert \geq \vert V(3r)\vert.  
\end{equation*}
\end{enumerate}
\end{lemma}

\begin{proof}
See \cite[Lemma~7.7.2]{ca-and-groups-springer} for the proof of (i). 
To show (ii), we define $V'$ to be the maximal subset of $V(3r)$ such that 
the balls $B(v,r)$ are pairwise disjoint for all $v \in V'$. 
If there exists $w \in V(3r) \setminus  \bigcup_{v\in V'} B(v,2r)$ 
then clearly $B(w,r) \cap B(v,r)$ is empty for all $v\in V$. 
Hence for all $v \in V'\cup \{w\}$, the balls $B(v,r)$ are pairwise disjoint, 
which contradicts the maximality of $V'$. 
We conclude that $V(3r) \setminus  \bigcup_{v\in V'} B(v,2r)$ is empty. 
\end{proof}

\section{Algebraic group cellular automata}
\label{s:alg-grp-ca}

\subsection{Interiors, neighborhoods, and boundaries}
\label{ss:interiors}
Let $G$ be a group and let $M\subset G$ be a finite subset. 
The $M$-\emph{interior} $\Omega^-$ and the $M$-\emph{neighborhood} $\Omega^+$ of a 
subset $\Omega \subset G$ are defined respectively by
\[
\Omega^- \coloneqq  \{g \in G : g M \subset \Omega \}, \quad \quad \Omega^ 
+ \coloneqq \Omega M = \{g h  : g \in \Omega \text{ and } h \in M \}. 
\]
Clearly, if $1_G \in M$ then $\Omega^- \subset \Omega \subset \Omega^+$. 
We denote also $\Omega^{++}\coloneqq (\Omega^+)^+$. 
The $M$-\emph{boundary} $\partial \Omega$ of $\Omega$ is defined as 
\[
\partial \Omega \coloneqq \Omega^+ \setminus \Omega^-.
\]
\par
If $G$ is an amenable group with a F\o lner net $(F_i)_{i \in I}$, 
then (cf.~\cite[Proposition~5.4.4]{ca-and-groups-springer}):
\begin{equation}
\label{e:boundary-folner}
\lim_i \frac{|\partial F_i |}{| F_i |} = 0 \quad \text{for every finite subset } M \subset G.
\end{equation}
\par
Suppose now that $\tau \colon A^G \to A^G$ is a cellular automaton over a group $G$ 
and an alphabet $A$ with a memory set $M$. 
Let $\Omega \subset G$ and $\Omega^-$ and $\Omega^+$ be as above. 
The induced maps
$\tau^-_\Omega \colon A^\Omega \to A^{\Omega^{-}}$ and 
$\tau^+_\Omega \colon A^{\Omega^{+}} \to A^\Omega$ are defined respectively by 
\[
\tau^-_\Omega(u) \coloneqq  (\tau(c))\vert_{\Omega^{-}}, \quad \text{ for all } u \in A^\Omega,
\]
and 
\[
\tau^+_\Omega(u) \coloneqq  (\tau(c))\vert_\Omega, \quad \text{ for all } u \in A^{\Omega^{+}},
\]
where $c \in A^G$ is any configuration extending $u$.
Note that $\tau^-_\Omega$ and $\tau^+_\Omega$ are well-defined 
as $\tau(c)(g)$ only depends of the restriction of $c$ to $g M$ by \eqref{e:local-property}. 
\par
For all subsets $\Gamma \subset A^G$ and $\Omega \subset G$, 
we define the \emph{restriction of $\Gamma$ to $A^\Omega$} by 
\[
\Gamma_\Omega\coloneqq \{c\vert_\Omega: c \in \Gamma\} \subset A^\Omega. 
\] 
If $\Gamma= \tau(A^G)$ then clearly $\Gamma_\Omega = \tau_\Omega^+(A^{\Omega^+})$ 
and $\Gamma_{\Omega^-} = \tau_\Omega^-(A^\Omega)$.

\subsection{Algebraic group cellular automata}
Let $K$ be an algebraically closed field. 
Let $G$ be a group and let $S$ be a scheme. 
Let $X,Y$ be $S$-schemes.
We denote by $X(Y)=\{S\text{-morphisms }h\colon Y \to X\}\eqqcolon A$ the set of $Y$-points of $X$.  
For every finite set $E$, the universal property of $S$-fibered products implies that
\[
(X^E)(Y) = (X(Y))^E.
\]
Each $S$-scheme morphism $f \colon Z \to X$ induces by composition a map
$f^{(Y)} \colon Z(Y) \to X(Y)$. 
\par
Let $\tau \colon A^G \to A^G$ be a cellular automaton over the alphabet $A$ and the group $G$.
We say that $\tau$ is an \emph{algebraic cellular automaton over the group $G$ and the schemes} $S$, $X$, $Y$ if 
$\tau$ admits a memory set $M$ with the associated local defining map $\mu_M \colon A^M \to A$ verifying: 

\begin{enumerate}
\item[(R)]
there exists an $S$-scheme morphism $f \colon X^M \to X$ such that $\mu_M = f^{(Y)}$.
\end{enumerate}

Let $k$ be a field (not necessarily algebraically closed). 
If $Y=S=\Spec(k)$, we say simply an \emph{algebraic cellular automaton over} $(G,X,k)$ 
instead of an \emph{algebraic cellular automaton over the group $G$ 
and the schemes $\Spec(k)$, $X$, $\Spec(k)$}. 
The set of algebraic cellular automata over $(G,X,k)$ is denoted by $\CA_{alg}(G,X,k)$.   

 \begin{remark}
 \label{rem:independent-memory}
 If $X(S) \not= \varnothing$, and condition (R) is satisfied for some memory set $M$ of $\tau$, 
 then (R) is satisfied for any memory set of $\tau$ (see \cite[Proposition~3.1]{cscp-alg-ca}).
This applies in particular in the case $S = \Spec(K)$ (and $X$ is nonempty). 
 \end{remark}

\begin{definition}
Let $G$ be a group. 
Let $X$ be an algebraic group over a field $k$ and $A \coloneqq X(k)$. 
\begin{enumerate} [\rm (1)]
\item
A cellular automaton $\tau \colon A^G \to A^G$ is called an 
\emph{algebraic group cellular automaton over} $(G,X,k)$
if for some memory set $M$ of $\tau$, 
there exists a homomorphism of $k$-algebraic groups $f \colon X^M \to X$ such that 
the map $f^{(k)} \colon A^M \to A$ is the local defining map of $\tau$ associated with $M$. 
\item
We denote by $\CA_{algr}(G,X,k) \subset \CA_{alg}(G,X,k)$ the set of algebraic group cellular automata over $(G,X,k)$. 
\end{enumerate}
\end{definition}

\begin{remark}
With the above notation, the configuration space $A^G$ is naturally an abstract group 
with identity element $e^G$ where $e\in A$ is the identity element of $X$. 
An algebraic group cellular automaton $\tau \colon A^G \to A^G$ is then 
an endomorphism of the abstract group $A^G$. 
The local defining maps of $\tau$ are also homomorphisms of abstract groups. 
\end{remark}

\begin{lemma}
\label{l:homo-alg}
Let $G$ be a group and let $X$ be an algebraic group over a field $k$. 
Let $A=X(k)$ and let $\tau \colon A^G \to A^G$ be an algebraic group cellular automaton. 
Let $M$ be a memory set of $\tau$ verifying (R).  
Then there exist canonical homomorphisms of $k$-algebraic groups 
\[
f^-_\Omega \colon X^\Omega \to X^{\Omega^-}, \quad 
f^+_\Omega \colon X^{\Omega^+} \to X^\Omega,  \quad
p_{\Omega \Lambda} \colon X^\Lambda \to X^\Omega
\] 
such that $f_\Omega^{-(k)}=\tau^-_\Omega$, $f_\Omega^{+(k)}=\tau^+_\Omega$ 
and $p^{(k)}_{\Omega \Lambda} \colon A^\Lambda \to A^\Omega$ is the canonical projection 
for all finite subsets $\Omega, \Lambda \subset G$ with $\Omega \subset \Lambda$. 
\end{lemma}

\begin{proof}
We have seen in Remark \ref{r:product-group} that $p_{\Omega \Lambda}$ 
is a $K$-homomorphism of algebraic groups for all finite subsets 
$\Omega, \Lambda \in G$ with $\Omega \subset \Lambda$. 
Since $f\colon X^M \to X$ is a homomorphism of algebraic groups, 
it follows again from Remark \ref{r:product-group} 
and the construction of the morphisms $f^-_\Omega,f^+_\Omega$ 
in \cite[Lemma~3.2]{cscp-alg-goe} that these morphisms 
are homomorphisms of algebraic groups which satisfy 
$f_\Omega^{-(k)}=\tau^-_\Omega$, $f_\Omega^{+(k)}=\tau^+_\Omega$ 
for every finite subset $\Omega \subset G$. 
\end{proof}

\begin{proposition}
\label{p:set-config-zariski}
Let $G$ be a group. 
Let $X$ be a $K$-algebraic group and $A \coloneqq X(K)$.  
Let $\tau \colon A^G \to A^G$ be an algebraic group cellular automaton over $(G,X,K)$.
Let $\Gamma \coloneqq \tau(A^G)$ denote the image of $\tau$. 
Then $\Gamma_\Omega$ is Zariski closed in $A^\Omega$ for every finite subset $\Omega \subset G$.
\end{proposition}

\begin{proof}
Let $M$ be a memory set of $\tau$ such that the associated local defining map
$\mu \colon A^M \to A$ is induced by some homomorphism  
$f \colon X^M \to X$ of algebraic groups.
Let $\Omega$ be a finite subset of $G$. 
Since homomorphisms of algebraic groups are closed by Theorem~\ref{t:homo-group}, 
it follows that  $f^+_\Omega \colon X^{\Omega^{+}} \to X^\Omega$ is closed.
We then get  
\[
\Gamma_\Omega = f_\Omega^{+}(A^{\Omega^+})=A^\Omega \cap f^+_\Omega(X^{\Omega^+}),  
\]
where the last equality follows from, for example, \cite[Proposition~2.10.(ii)]{cscp-alg-goe} 
and the identification of $A^{\Omega^+}$ with the set of closed point of $X^{\Omega^+}$. 
Hence, $\Gamma_\Omega$ is closed in $A^\Omega$. 
\end{proof}

\subsection{Algebraic mean dimension}
\label{s:alg-mean-dim}
To establish a version of the Garden of Eden theorem in Section~\ref{s:goe-grp}, 
we shall need the following useful notion of mean dimension introduced 
in \cite{cscp-alg-goe} when the group $G$ is amenable. 

\begin{definition} (cf. \cite{cscp-alg-goe})
\label{def:mean-dim}
Let $G$ be an amenable group with a F\o lner net $\FF=(F_i)_{i \in I}$.   
Let $X$ be a $K$-algebraic variety and $A \coloneqq X(K)$. 
The \emph{algebraic mean dimension} of a subset
$\Gamma \subset A^G$ with respect to  $\FF$ 
is 
defined by
\begin{equation}  
\label{e;mdim}
\mdim_\FF(\Gamma) \coloneqq  \limsup_{i \in I} \frac{\dim(\Gamma_{F_i})}{| F_i |},
\end{equation}
where $\dim(\Gamma_{F_i})$ is the Krull dimension of 
$\Gamma_{F_i} \subset A^{F_i} \subset X^{F_i}$ with respect to the Zariski topology. 
\end{definition}

We see immediately that (cf.~\cite[Proposition~4.2]{cscp-alg-goe})

\begin{enumerate}[{\rm (i)}]
\item
$0 \leq \mdim_\FF(\Gamma)  \leq \dim(X)$ for every nonempty subset $\Gamma \subset A^G$;  
\item 
$\mdim_\FF(A^G) = \dim(X)$;
\item
$\mdim_\FF(\Gamma)\leq \mdim_\FF(\Gamma')$ for all subsets $\Gamma \subset \Gamma' \subset A^G$.  
\end{enumerate}

For further properties of algebraic mean dimension, see \cite{cscp-alg-goe}. 

\section{A criterion for nonemptyness of inverse limits of algebraic varieties} 
\label{s:criterion-inverse}

In this section, we establish a sufficient condition for nonemptyness 
of inverse limits of algebraic varieties which will be used in the next sections. 
Following the idea suggested in the remark after the proof of~\cite[Theorem~1]{stone}, 
we obtain in Proposition~\ref{p:stone-general} a stronger statement than   
\cite[Theorem~1]{stone} in the case of algebraic varieties. 
We begin by giving a proof of the following well-known result: 

\begin{lemma}
\label{l:C}
Let $(S_\lambda)_{\lambda \in \Lambda}$ be a decreasing directed family 
of closed subsets of a topological space $X$. 
Let $f \colon X\to Y$ be a continuous map such that $f^{-1}(y)$ is quasi-compact for all $y\in Y$. 
Then we have 
\begin{equation*}
f(\cap_{\lambda \in \Lambda}S_\lambda)=\cap_{\lambda \in \Lambda} f(S_\lambda). 
\end{equation*}
\end{lemma} 
 
\begin{proof}
Observe first that $f(\cap_{\lambda \in \Lambda}S_\lambda)\subset \cap_{\lambda \in \Lambda} f(S_\lambda)$. 
Now let $y\in \cap_{\lambda \in \Lambda} f(S_\lambda)$. 
For each $\lambda \in \Lambda$, we define a nonempty closed subset $F_\lambda\coloneqq f^{-1}(y) \cap S_\lambda$ of 
$f^{-1}(y)$. 
Since these closed subsets $F_\lambda$ ($\lambda \in \Lambda$) satisfy the finite-intersection property, 
their intersection is nonempty by quasi-compactness of $f^{-1}(y)$. 
Hence,   
\begin{equation}
\label{e:C}
f^{-1}(y) \cap (\cap_{\lambda \in \Lambda} S_\lambda)=\cap_{\lambda \in \Lambda} F_\lambda \neq \varnothing.
\end{equation}
Therefore, there exists $x\in \cap_{\lambda \in \Lambda} S_\lambda$ such that $y=f(x)$. 
It follows that $f(\cap_{{\lambda \in \Lambda}}S_\lambda)\supset \cap_{\lambda \in \Lambda} f(S_\lambda)$ 
and thus $f(\cap_{{\lambda \in \Lambda}}S_\lambda)=\cap_{\lambda \in \Lambda} f(S_\lambda)$. 
This completes the proof of the lemma. 
\end{proof}

\begin{proposition}
\label{p:stone-general} 
Let $(X_\lambda, \varphi_{\nu \lambda})_{ \Lambda}$ be a countable 
inverse directed system of $K$-algebraic varieties. 
For all $\nu \geq \lambda$ in $\Lambda$, suppose that 
$\varphi_{\lambda \nu} \colon X_\nu \to X_\lambda$ is a morphism of $K$-schemes 
such that $\varphi_{\lambda \nu}(X_\nu)\subset X_\lambda$ is a closed subset. 
Then the induced inverse limit on $K$-points $\varprojlim_{\lambda \in \Lambda} X_\lambda(K)$ is nonempty.  
\end{proposition}

\begin{proof}
For each $\lambda \in \Lambda$, let us define a subset $Y_\lambda$ of $X_\lambda(K)$ by 
\[
Y_\lambda\coloneqq \cap_{\nu \geq \lambda }\, \varphi_{\lambda \nu}(X_\nu(K)).
\] 
Note that $X_\nu(K)$ is nonempty since $X_\nu$ is nonempty for every $\nu \in \Lambda$. 
By hypothesis, we see that $(\varphi_{\lambda \nu}(X_\nu(K)))_{\nu \geq \lambda}$ 
forms a decreasing family of nonempty closed subsets of $X_\lambda(K)$. 
Hence, $Y_\lambda \subset  X_\lambda(K)$ is a nonempty closed subset since 
$X_\lambda(K)$ is quasi-compact for all $\lambda \in \Lambda$. 
For $\mu \geq \lambda$ in $\Lambda$, we consider the induced continuous map  
\[
f_{\lambda \mu} \coloneqq \varphi_{\lambda \mu} \vert_{Y_\mu}\colon Y_\mu \to Y_\lambda. 
\]
Let $y\in Y_\lambda$  
then $f^{-1}(y)=\varphi^{-1}(y)\cap Y_\mu$ is a closed subset of $X_\mu(K)$ hence quasi-compact. 
It follows from Lemma~\ref{l:C} that we have
\begin{align*}
f_{\lambda \mu}(Y_\mu)
& = \varphi_{\lambda \mu} \left(\cap_{\nu \geq \mu } \varphi_{\mu \nu}(X_\nu(K)) \right) \\
& = \cap_{\nu \geq \mu} \,  (\varphi_{\lambda \mu}\circ  \varphi_{\mu \nu} ) (X_\nu(K)) \\
& =  \cap_{\nu \geq \mu}  \, \varphi_{\lambda \nu}(X_\nu(K))\\
&= \cap_{\nu \geq \lambda}  \, \varphi_{\lambda \nu}(X_\nu(K))= Y_\lambda, 
\end{align*}
and thus $f_{\lambda \mu}$ is surjective for all $\mu \geq \lambda$ in $\Lambda$. 
Therefore, we obtain an inverse subsystem $(Y_\nu,f_{\lambda \nu})$ of nonempty spaces 
in which the transition maps are surjective. 
As the system is countable, it follows that $\varprojlim_{\lambda \in \Lambda} Y_\lambda$ is nonempty. 
To see this, we reduce to the obvious case where the index set $\Lambda$ is $\N$ 
with the natural order as follows. 
As $\Lambda$ is countable, we can write $\Lambda= \cup_{n\in \N} A_n$ 
where $(A_n)_{n\in \N}$ is an increasing sequence of finite subsets. 
Then define by induction an increasing sequence $(a_n)_{n \in \N}$ in $\Lambda$ 
such that for every $n\in N$, $a_n \geq \lambda$ for all $\lambda \in A_n$. 
This is possible since every $A_n$ is finite and $\Lambda$ is directed. 
It is not hard to see that $\varprojlim_{\lambda \in \Lambda} Y_\lambda=\varprojlim_{n \in \N} Y_{a_n}$, 
which is now obviously nonempty. 
\par      
By construction, we have $\varprojlim_{\lambda \in \Lambda} X_\lambda(K)=\varprojlim_{\lambda \in \Lambda} Y_\lambda$. 
We can thus conclude that $\varprojlim_{\lambda \in \Lambda} X_\lambda(K)$ is nonempty.  
\end{proof}

\section{The closed image property in the prodiscrete topology}
\label{s:cip-grp}

The goal of this section is to establish a fundamental property on the image of algebraic group cellular automata, 
namely, the closed image property with respect to the prodiscrete topology. 
The following theorem extends \cite[Theorem~4.1]{cscp-alg-ca} to the class of algebraic groups. 

\begin{theorem}
\label{t:closed-image}
Let $G$ be a group. 
Let $X$ be a $K$-algebraic group and $A \coloneqq X(K)$. 
Let $\tau \colon  A^G\to A^G$ be an algebraic group cellular automaton over $(G,X,K)$.  
Then $\tau(A^G)$ is closed in $A^G$ for the prodiscrete topology. 
\end{theorem}

\begin{proof} 
We denote the group law (not necessarily commutative) of algebraic groups by $+$. 
\par
By hypothesis, $\tau$ admits a memory set $M$ whos local defining map $\mu \colon A^M \to A$ 
is induced by a homomorphism of algebraic groups $f\colon X^M \to X$. 
By Lemma \ref{l:homo-alg}, we see that $f^-_\Omega \colon X^\Omega  \to X^{\Omega^-}$ 
and the projection $p_{\Omega \Lambda} \colon X^\Lambda \to X^\Omega$ are homomorphisms of $K$-algebraic groups 
for all finite subsets $ \Omega, \Lambda \subset G$ such that $\Omega \subset \Lambda$.
\par 
Let $d\in A^G$ be a configuration that belongs to the closure of $\tau(A^G)$ in $A^G$ with respect to the prodiscrete topology. 
Let $\PP$ denote  the directed set consisting of all finite subsets of $G$ partially orderable by inclusion. 
Consider the inverse directed system $(Y_\Omega, \phi_{\Omega  \Lambda})_{\PP}$ 
of nonempty $K$-algebraic varieties defined by: 
\[
Y_\Omega \coloneqq (f^-_\Omega)^{-1}(d\vert_{\Omega^-})\subset X^\Omega \quad \text{and} \quad
\phi_{\Omega  \Lambda} \coloneqq p_{\Omega \Lambda}\vert_{Y_\Lambda} \colon Y_\Lambda  \to Y_\Omega.  
\]
Each $Y_\Omega \subset X^\Omega$ is equipped with the induced reduced closed subscheme structure 
and the $K$-scheme morphism $\phi_{\Omega  \Lambda}$ is uniquely induced by the projection 
$p_{\Omega \Lambda}$ (cf. \cite[Proposition~I.5.2.2]{grothendieck-ega-1}). 
Remark that by the choice of $d$, the $K$-algebraic varieties $Y_\Omega$ are nonempty.   
\par
We claim that $\phi_{\Omega \Lambda}(Y_\Lambda)$ is a closed subset of $Y_\Omega$ 
for all finite subsets $\Omega,\Lambda \in G$ with $\Omega \subset \Lambda$. 
Indeed, consider $I_\Omega\coloneqq \Ker(f^-_\Omega) \subset X^\Omega$.  
By Theorem \ref{t:homo-group}, we see that $I_\Omega$ is a nonempty closed algebraic subgroup of $X^\Omega$.  
Let $c_\Omega\in Y_\Omega$ be a closed point. 
Then $Y_\Omega=c_\Omega + I_\Omega$. 
Since $p_{\Omega \Lambda}$ is a homomorphism, we deduce that 
\[
p_{\Omega \Lambda}(Y_\Lambda)=p_{\Omega \Lambda}(c_\Lambda)+p_{\Omega \Lambda}(I_\Lambda).
\] 
It follows from Theorem \ref{t:homo-group} that $p_{\Omega \Lambda}(I_\Lambda)$ 
is a closed algebraic subgroup. 
Hence, $\phi_{\Omega \Lambda}(Y_\Lambda)=p_{\Omega \Lambda}(Y_\Lambda)$ 
is also a closed subset of $Y_\Omega$. This proves the claim. 
\par
Let $H\subset G$ be the subgroup generated by $M$. 
Let $\tau_H \colon A^H \to A^H$ be the restriction of $\tau$ to $A^H$. 
By \cite[Theorem~1.2.(vi)]{csc-ind-res}, $\tau(A^G)$ is closed in $A^G$ for the prodiscrete topology if and only if 
$\tau_H(A^H)$ is closed in $A^H$ for the prodiscrete topology. 
Up to replacing $G$ by $H$ and $\tau$ by $\tau_H$, we can thus suppose that $G$ is countable.  
It follows from Proposition \ref{p:stone-general} 
that the induced inverse limit of $K$-points $\varprojlim_{\Omega\in \PP} Y_\Omega(K)$ is nonempty.  
Observe finally that by~\cite[Lemma~2.1]{cscp-alg-ca}, the cellular automata $\tau$ has the closed image property 
if and only if $\varprojlim_{\Omega \in \PP} Y_{\Omega}(K)$ is nonempty.  
This completes the proof of Theorem \ref{t:closed-image}. 
\end{proof}

As a first application of Theorem \ref{t:closed-image}, we obtain the following: 
\begin{theorem}
\label{t:surj-grp}
Let $G$ be a locally residually finite group. 
Let $X$ be a $K$-algebraic group and $A \coloneqq X(K)$. 
Let $\tau \colon A^G\to A^G$ be an algebraic group cellular automaton over $(G,X,K)$. 
Suppose that $\tau$ is injective. 
Then $\tau$ is surjective and hence bijective. 
\end{theorem}

\begin{proof}
The proof of this theorem is identical, \emph{mutatis mutandis}, to the proof of \cite[Theorem~1.2]{cscp-alg-ca}. 
Note that $\tau$ satisfies the closed image property by Theorem \ref{t:closed-image}. 
\end{proof}

\section{Reversibility and Invertibility of algebraic group cellular automata}
\label{s:reverse-inverse}

\subsection{Left inverse of injective algebraic group cellular automata}

We begin with an auxiliary lemma on the reversibility of injective algebraic group cellular automata.  

\begin{lemma}
\label{l:techno}
Let $G$ be a group. 
Let $X$ be a $K$-algebraic group and $A\coloneqq X(K)$. 
Let $\tau \colon A^G\to A^G$ be an injective algebraic group cellular automaton over $(G,X,K)$. 
Let $\Gamma\coloneqq \tau(A^G)$ denote the image of $\tau$.
Then there exists a finite subset $N\subset G$ such that 
\begin{enumerate}
\item[(C)]  
 \textit{for any} $d\in \Gamma$, the element  
 $\tau^{-1}(d)(1_G)\in A$ \textit{depends only on} $d \vert_N \in A^N$.
\end{enumerate}
\end{lemma}

\begin{proof}
Let $M\subset G$ be a memory set  
of $\tau$ such that $1_G \in M$.    
Let $H \subset G$ be the subgroup generated by $M$. 
Then $H$ is countable 
and thus admits an increasing sequence of finite subsets 
$M=E_0 \subset \dots \subset E_n \dots$~such that $H=\cup_{n\in \N} E_n=\cup_n E^-_n$. 
As $\tau \colon A^G \to A^G$ is an injective homomorphism of abstract groups, 
$\tau^{-1} \colon \Gamma \to A^G$ is also a homomorphism of abstract groups. 
\par
Suppose on the contrary that there does not exist a finite subset $N\subset G$ verifying $(C)$. 
Hence, for each $n \in \N$, there exist configurations $d_n,d'_n \in \Gamma$ such that 
\begin{equation*}
d_n\vert_{E_n^-}=d'_n\vert_{E_n^-} \quad \text{and} \quad \tau^{-1}(d_n)(1_G)\neq \tau^{-1}(d'_n)(1_G). 
\end{equation*}
Let $c'_n\coloneqq d_n - d'_n \in A^G$ for each $n \in \N$. 
Then we find that 
$c'_n\vert_{E_n^-}=e^{E_n^-}$ and $\tau^{-1}(c'_n)(1_G)\neq e$.
Thus by setting $c_n\coloneqq \tau^{-1}(c'_n)\vert_{E_n}\in A^{E_n}$ for every $n \in \N$, we obtain:  
\begin{equation}
\label{e:condition}
\tau_{E_n}(c_n)=c'_n\vert_{E_n^-}= e^{E_n^-}, \quad  \quad c_n(1_G) \neq e. 
\end{equation}
For each $n\in \N$, we define a closed subgroup of $A^{E_n}$ (by Theorem \ref{t:homo-group})
\[
I_n\coloneqq \Ker(\tau^-_{E_n})=\left(\Ker(f^-_{E_n})\right) (K) \subset A^{E_n}.
\]
Then $I_n$ is nonempty for all $n\in \N$ as $c_n \in I_n$ by \eqref{e:condition}. 
The projection 
$p_{nm} \colon A^{E_m} \to A^{E_n}$ induces a homomorphism of groups 
$\pi_{nm} \colon I_m \to I_n$ for all $n \leq m$. 
It is clear that 
\[
\pi_{nm'}(I_{m'}) \subset \pi_{nm}(I_m) \quad \text{whenever } m' \geq m.
\] 
Thus, for each $n\in \N$, we obtain by Theorem \ref{t:homo-group} a decreasing sequence 
of closed subgroups $(\pi_{nm}(I_m))_{m \geq n}$ of $I_n$. 
As $\Ker(f^-_{E_n})$ is an algebraic variety hence Noetherian, $I_n$ is also Noetherian. 
Hence, $(\pi_{nm}(I_m))_{m \geq n}$ is stationary and 
there exists a nonempty closed subgroup $J_n \subset I_n\subset A^{E_n}$ 
such that $\pi_{nm}(I_m)=J_n$ for all $m$ large enough. 
\par
It is clear that $\pi_{nm}$ induces by restriction a homomorphism of groups 
$q_{nm} \colon J_m  \to J_n$ for all $m \geq n$. 
We claim that $q_{nm} \colon J_m  \to J_n$ is surjective for all $m\geq n$. 
Indeed, let $y\in J_n$ and $k \geq m$ be large enough 
such that $q_{n k}(I_k)=J_n$ and $q_{m k}(I_k)=J_m$. 
Hence, there exists $x\in I_k$ such that $q_{nk}(x)=y$. 
Since $q_{nk}= q_{nm} \circ q_{mk}$, we see that $q_{nm}(y')=y$ where $y'=q_{mk}(x) \in J_m$. 
This proves the claim.
\par 
Now let $k\in \N$ be such that $\pi_{0k}(I_k)=J_0$ and let $x_0=\pi_{0k}(c_k) \in J_0$. 
From \eqref{e:condition}, we see that $x_0(1_G) \neq e$. 
We construct by induction a sequence $(x_n)_{n \in \N}$ where 
$x_n \in J_n$ for all $n\in \N$ as follows. 
Given $x_n \in J_n$ for some $n\in \N$, we can then choose 
by the surjectivity of the map $q_{n,n+1}$ an element 
\[
x_{n+1}\in q^{-1}_{n,n+1}(x_n) \subset J_{n+1}\subset A^{E_{n+1}}.
\] 
Since $q_{nr}= q_{nm} \circ q_{mr}$ for all $n\leq m \leq r$ and 
$H=\cup_n E_n$, there exists $a_H \in A^H$ such that 
\[
a_H\vert_{E_n}=x_n \quad \text{for all } n\in \N.
\] 
\par
Now let $R \subset G$ be a set of representatives of the quotient $G/ H$ such that $1_G\in R$. 
For each $\delta \in R$, define $a_{\delta H} \in A^{\delta H}$ 
by $a_{\delta H}(g)\coloneqq a_H(\delta^{-1}g)$ for all $g\in \delta H$.  
We obtain a configuration:  
\begin{equation}
\label{e:tecnic}
a\coloneqq \prod_{\delta \in R} a_{\delta H} \in \prod_{\delta \in R} A^{\delta H}= A^G. 
\end{equation}
Since $M\subset H$, we deduce from \eqref{e:tecnic} and \eqref{e:local-property} 
that for all $\delta \in R$ and $h\in H$, we have 
\begin{equation}
\label{e:tecta}
(\tau(a))(\delta h)=(\tau(a))(h). 
\end{equation}
 
Observe that $a(1_G)=x_0(1_G)\neq e$ hence $a\neq e^{G}$. 
\par
We see by construction that for all $n \in \N$, 
\[
\tau(a)\vert_{E_n^-}=\tau^-_{E_n}(a\vert_{E_n})=\tau^-_{E_n}(x_n)=e^{{E_n^-}}  
\quad \left(\text{since }x_n \in J_n \subset I_n =\Ker(\tau^-_{E_n}) \right). 
\]
Hence, $\tau(a)\vert_H=e^{H}$ as $H=\cup_n E^-_n$. 
From \eqref{e:tecta}, it follows that $\tau(a)\vert_{\delta H}=e^{\delta H}$ for all $\delta \in R$. 
Therefore, $\tau(a)= e^{G}$, which contradicts the injectivity of $\tau$ 
since $a\neq e^G$ and $\tau(e^G)=e^G$.  
This completes the proof of Proposition \ref{p:left-inverse}. 
\end{proof}

We can now establish the following key technical result on left inverses 
of injective algebraic group cellular automata over any universe $G$. 
The main point is that the local defining map of the inverse map 
of an injective algebraic group cellular automaton 
is also induced by a homomorphism of algebraic groups. 

\begin{proposition}
\label{p:left-inverse}
Suppose that $\car(K)=0$. Let $G$ be a group. 
Let $X$ be an $K$-algebraic group and $A\coloneqq X(K)$. 
Let $\tau \colon A^G\to A^G$ be an injective algebraic group cellular automaton over $(G,X,K)$. 
Let $M$ be a memory set of $\tau$ with $1_G \in M$ and let $f \colon X^M \to X$ 
be a homomorphism of $K$-algebraic groups such that $f^{(K)} \colon A^M \to A$ 
is the associated local defining map of $\tau$.  
Then there exists a finite subset $N \subset G$ satisfying the following property
\begin{enumerate}
\item[(P)]
For every finite subset $ \Omega \subset G$ containing $N$, 
there exists a homomorphism of $K$-algebraic groups 
$h_\Omega \colon  Y_\Omega \to X$ 
where $Y_\Omega \coloneqq f^+_\Omega(X^{\Omega^+})\subset X^\Omega$ 
and that for all $c\in A^G$:  
\begin{equation}
\label{e:local-inverse-1}
c(1_G)=h_\Omega^{(K)}(\tau(c)\vert_\Omega). 
\end{equation}
\end{enumerate}
\end{proposition}

\begin{proof}  
Let $\Gamma = \tau(A^G)$. 
Group laws of algebraic groups are denoted by $+$. 
By Lemma~\ref{l:techno}, there exists a finite subset $N\subset G$ such that 
 the following condition holds:
\begin{enumerate}
\item[(C)]  
 \textit{For any} $d\in \Gamma$, the element  
 $\tau^{-1}(d)(1_G)\in A$ \textit{depends only on} $d \vert_N \in A^N$.
\end{enumerate}
Up to enlarging $N$, we can assume that $1_G \in N$. 
Let $\Omega \subset G$ be a finite subset containing $N$. 
Theorem \ref{t:homo-group} and Lemma \ref{l:homo-alg} imply that 
$Y_\Omega= f^+_\Omega(X^{\Omega^+})$ is a closed $K$-algebraic subgroup of $X^\Omega$. 
Similarly, the kernel of the homomorphism $f^+_\Omega$
\[
U_\Omega \coloneqq \Ker(f^+_\Omega) \subset X^{\Omega^+}
\]
 is a closed normal $K$-algebraic subgroup. 
Let $y\in \Gamma_\Omega= Y_\Omega(K)$ 
and $x_0 \in (\tau^+_\Omega)^{-1}(y)\subset A^{\Omega^+}$, we find that 
\begin{equation}
\label{e:tec}
(\tau^+_\Omega)^{-1}(y)=x_0 + U_\Omega(K). \end{equation}
Hence, $x(1_G)$ is constant for all $x\in (\tau^+_\Omega)^{-1}(y) \subset A^{\Omega^+}$. 
Indeed, let $x\in  (\tau^+_\Omega)^{-1}(y)$ and $\tilde{x}, \tilde{x}_0\in A^G$ 
be any configurations that extend $x$ and $x_0$ respectively. 
Since $N \subset \Omega$, we have  
\[
\tau(\tilde{x})\vert_N=\tau_\Omega^+(x)\vert_N=y\vert_N=\tau_\Omega^+(x_0)\vert_N=\tau(\tilde{x}_0)\vert_N.
\] 
Hence, Condition $(C)$ applied for $\tau(\tilde{x}), \tau(\tilde{x_0}) \in \Gamma$ implies that 
\[
x(1_G)=\tilde{c}(1_G)=(\tau^{-1}\tau(\tilde{x}))(1_G)=(\tau^{-1}\tau(\tilde{x}_0))(1_G)=\tilde{x}_0(1_G)=x_0(1_G).
\]
Consider the projection $\rho \colon X^{\Omega^+} \to X^{\{1_G\}}$; then 
it follows that $\rho^{(K)}((\tau^+_\Omega)^{-1}(y))=x_0(1_G)$. 
We combine this with \eqref{e:tec} to obtain: 
\[
x_0(1_G)= \rho^{(K)}(x_0+U_\Omega(K))= x_0(1_G) + \rho^{(K)}(U_\Omega(K)).
\]
Hence, $\rho^{(K)}(U_\Omega(K))=\{e\}$. 
Identifying $U_\Omega(K)$ with the dense subset consisting of closed points of $U_\Omega$, 
we deduce that $\rho(U_\Omega)=\{e\}$ by the continuity of $\rho$ and the fact that $e$ is a closed point. 
From the universal property of the 
 quotient map $\pi \colon X^{\Omega^+} \to X^{\Omega^+}/ U_\Omega \simeq Y_\Omega$ 
 induced by $f^+_\Omega$ 
(cf. Theorem \ref{t:homo-group} and Theorem \ref{t:quotient-universal}), we obtain 
a unique homomorphism of $K$-algebraic groups $h_\Omega\colon Y_\Omega \to X^{\{1_G\}}$ 
which makes the following diagram commutative:
\[
\begin{tikzcd}
X^{\Omega^+} \arrow[d,"\pi"]   \arrow[r,"\rho"] & X^{\{1_G\}} \\
X^{\Omega^+}/ U_\Omega \arrow[r, " \simeq"] & Y_\Omega \arrow[u, "h_\Omega",swap]
\end{tikzcd}
\]
Now for every $c \in A^G$, we deduce immediately from the diagram that 
\[
h^{(K)}_\Omega\left( \tau(c)\vert_\Omega \right) 
= h^{(K)}_\Omega\left( \tau^+_\Omega(c\vert_{\Omega^+}) \right)
= \rho^{(K)} (c\vert_{\Omega^+}) =c(1_G). 
\] 
This completes the proof of Proposition~\ref{p:left-inverse}.  
\end{proof} 

The proofs of Lemma~\ref{l:techno} and Proposition~\ref{p:left-inverse} can be easily adapted 
to obtain the following reversibility result:  

\begin{proposition}
\label{p:left-inverse-2}
Let $G$ and $A$ be groups. 
Let $\tau \colon A^G\to A^G$ be a group cellular automaton 
over the universe $G$ and the alphabet $A$, i.e, 
$\tau$ admits a local defining map which is a homomorphism of abstract groups. 
Let $\Gamma \coloneqq \im(\tau)$. 
Suppose that $A$ is an Artinian group and $\tau$ is injective. 
Then there exists a finite subset $N \subset G$ such that 
\begin{enumerate}
\item[(D)]
For every finite subset $ \Omega \subset G$ containing $N$, 
there exists a homomorphism of groups 
$h_\Omega \colon  \Gamma_\Omega \to A$
and that for all $c\in A^G$:  
\begin{equation}
\label{e:local-inverse-1}
c(1_G)=h_\Omega(\tau(c)\vert_\Omega). 
\end{equation}
\end{enumerate}

\end{proposition}

\subsection{Invertibility of bijective algebraic group cellular automata}

As a direct application of Proposition \ref{p:left-inverse}, we deduce 
the following result which extends \cite[Theorem~1.3]{cscp-alg-ca}: 

\begin{theorem}
\label{t:inverse-also}
Let $G$ be a group and let $X$ be a $K$-algebraic group. 
Suppose that $\car(K)=0$. 
Let $A \coloneqq X(K)$ and $\tau \in CA_{algr}(G,X, K)$ be such that $\tau \colon A^G\to A^G$ is bijective. 
Then the inverse map $\tau^{-1}\colon A^G\to A^G$ is also an element of $CA_{algr}(G,X,K)$.  
\end{theorem}

\begin{proof} 
It suffices to take $\Gamma=A^G$ and $N$ (as in Proposition \ref{p:left-inverse}) 
to be the memory set of $\tau^{-1}\colon A^G \to A^G$. 
Then by the same proposition, the homomorphism of $K$-algebraic groups 
$h_N \colon X^\Omega \to X$ induces the associated local defining map of $\tau^{-1}$. 
\end{proof}

\begin{proof}[Proof of Theorem~\ref{t:surjunctive-res}] 
It suffices to combine Theorem~\ref{t:inverse-also} and Theorem~\ref{t:surj-grp}. 
\end{proof}

The following proposition implies a certain stability property of units of $CA_{algr}(G,X,K)$ in $CA_{alg}(G,X,K)$.  

\begin{proposition}
\label{t:inverse-also-near-ring-1}
Let $G$ be a group and let 
$X$ be a $K$-algebraic group. 
Let $A \coloneqq X(K)$ and $\tau, \sigma \in CA_{alg}(G,X,K)$. 
Suppose that $\sigma\circ \tau= \tau \circ \sigma=\Id_{A^G}$ and that $\tau \in CA_{algr}(G,X,K)$. 
Then $\sigma \in CA_{algr}(G,X,K)$.  
\end{proposition}

\begin{proof}
Choose any memory set $M$ of $\sigma$ and consider 
the associated local defining map $\mu \colon A^M \to A$. 
Since $\sigma \in CA_{algr}(G,X,K)$ and $K$ is algebraically closed, 
there exists a morphism of $K$-algebraic varieties $f \colon X^M \to X$ 
such that $f^{(K)}=\mu$ (Remark \ref{rem:independent-memory}). 
As $\tau$ is a homomorphism of abstract groups by hypothesis, 
so is $\sigma \colon A^G \to A^G$. 
It follows that $f^{(K)}=\mu$ is also a homomorphism of groups. 
Since $X^M$ and $X$ are reduced, separable and $K$ is algebraically closed, 
$f$ is in fact a homomorphism of algebraic groups. 
Therefore, $\sigma \in CA_{algr}(G,X,K)$ and this completes 
the proof of Proposition \ref{t:inverse-also-near-ring-1}. 

\end{proof}

\section{Surjunctivity over sofic groups}
\label{s:sofic}

We have proved the surjunctivity of algebraic group cellular automata over $(G,X,K)$ 
with $G$ a locally residually finite group (Theorem~\ref{t:surjunctive-res}). 
When the algebraic group $X$ is connected, 
this property holds more generally with $G$ a sofic group. 

\begin{theorem}
\label{t:sofic-surjunctive}
Let $G$ be a sofic group. 
Let $X$ be an connected $K$-algebraic group and $A=X(K)$. 
Suppose that $\tau \in CA_{algr}(G,X,K)$ is injective. 
Then $\tau$ is surjective and hence bijective. 
\end{theorem}

\begin{proof}
If $\dim(X)=0$, then $X$ is a singleton since it is connected and 
the theorem is trivial. 
\par 
We suppose now that $n\coloneqq \dim(X)>0$. 
Let $M$ be a memory set of $\tau$ such that $1_G \in M$. 
Let $H\subset G$ be the subgroup generated by $M$. 
Then $H$ is finitely generated. 
We see that (cf.~\cite[Proposition~1.7.4]{ca-and-groups-springer}) 
$\tau$ is surjective if and only if the restriction 
$\tau_H \colon A^H \to A^H$ of $\tau$ to $A^H$ is surjective. 
As $\tau$ is injective, $\tau_H$ is clearly injective. 
Note also that every subgroup of a sofic group is sofic (cf.~\cite[Proposition~7.5.4]{ca-and-groups-springer}).  
Thus, up to replacing $G$ by $H$, we can suppose that 
$G$ is finitely generated and hence countable. 
\par
Let $\Gamma \coloneqq \tau(A^G)$ and let $M$ be a fixed memory set of $\tau$. 
Let $S \subset G$ be a finite symmetric generating subset of $G$.  
For $s \in \N$, recall that $B_S(s) \subset G$ is the ball of
radius $s$ centered at 
$1_G$ in the Cayley graph $C_S(G)$ (cf. Subsection~\ref{ss:sofic-groups}). 
\par
Since $\tau$ is injective, Proposition~\ref{p:left-inverse} implies that 
there exists a nonempty finite subset $N\subset G$ 
such that for every finite subset $\Omega \subset G$ containing $N$, 
we have a homomorphism of $K$-algebraic groups 
$h_\Omega \colon  Y_\Omega \to X$ where $Y_\Omega\coloneqq  f^+_\Omega(X^{\Omega^+})$ 
such that for all $c\in A^G$:  
\begin{equation}
\label{e:local-inverse}
c(1_G)=h_\Omega^{(K)}(\tau(c)\vert_\Omega)=h_\Omega^{(K)}(\tau^+_\Omega(c\vert_{\Omega^+})) .
\end{equation}
\par
Choose $r \in \N$ large enough such that  
$B_S(r) \supset M \cup N$.  
Up to enlarging $M$ and $N$, we can suppose that $M=N=B_S(r)$. 
Let $\mu \colon A^M \to A$ be the local defining map of $\tau$ associated to $M$. 
We denote also $\eta \coloneqq h_M^{(K)}  \colon \Gamma_M \to A$. 
\par
We suppose on the contrary that $\tau$ is not surjective, i.e., $\Gamma \subsetneq A^G$. 
It follows from Theorem~\ref{t:closed-image} that 
$\Gamma$ is closed in $A^G$ with respect to the prodiscrete topology.  
Therefore, there exists a finite subset $\Lambda \subset G$ 
such that $\Gamma\vert_\Lambda \subsetneq A^\Lambda$. 
Up to enlarging $r$ and $M$ again, we can suppose that $\Lambda \subset B_S(r)=M$. 
It follows that 
\begin{equation}
\label{e:contrary}
\Gamma \vert_{B_S(r)} \subsetneq A^{B_S(r)}. 
\end{equation}
\par
Now we choose $\varepsilon \in \R $ satisfying 
\begin{equation*} \label{e:contra-hypothesis} 
   0 < \varepsilon < \frac{1}{n \vert B_S(2r) \vert+1}, 
\end{equation*}
where we recall that $n= \dim(A)$. 
Then it follows immediately that 
\begin{equation} \label{e:contra-hypothesis-1}
   0<   (1 - \varepsilon)^{-1} < 1 + \frac{1}{n \vert B_S(2r) \vert}.
\end{equation}
\par
Since the group $G$ is sofic, it follows from Theorem~\ref{t:sofic-character} 
that there exists a finite $S$-labeled graph $\GG=(V,E)$ associated to the pair $(r, \varepsilon)$ 
such that
\begin{equation} 
\label{e:sofic-V}
    \vert V(3r) \vert \geq (1 - \varepsilon) \vert V \vert.
\end{equation}
Recall that for each $s \in \N$, the finite subset $V(s)\subset V$ 
consists of $v \in V$ such that there exists a unique $S$-labeled graph
isomorphism $\psi_{v,s} \colon B_S(s) \to B_\GG(v,s)$ satisfying 
$\psi_{v,s}(1_G) = v$ (cf. Theorem \ref{t:sofic-character}).
Note that $V(3r) \subset V(2r) \subset V(r)\subset V$. 
\par
To simplify the notation, we denote $B(v,s)\coloneqq B_\GG(v,s)$ for all $v \in V$ and $s\in \N$.  
By Lemma~\ref{l:sofic-B-V}.(i), we have for all $v \in V((k+1)r)$ where $k \in \{1,2\}$ that 
\begin{equation}
\label{e:inclusions}
B(v,r) \subset V(kr) \subset V. 
\end{equation}
Therefore, we can consider the homomorphism $\mu_1 \colon A^{V(r)} \to A^{V(2 r)}$ given by 
\[
\mu_1(c)(v) \coloneqq \mu\left(c\vert_{B(v,r)} \circ \psi_{v,r}\right)
\]
for all $c \in A^{V(r)}$ and $v \in V(2r)$. 
\par
Then $\mu_1$ is induced by a homomorphism of algebraic groups. 
Indeed, we consider for each $v \in V(r)$ the isomorphism of algebraic groups
\[
\beta_v \colon  X^{B(v,r)} \to X^{B_S(r)}=X^M
\]
induced by the bijection $\psi_{v,r}\colon B_S(r) \to B(v,r)$. 
Let $\pi_{1,v} \colon X^{V(r)} \to X^{B(v,r)}$ be the projection for $v \in V(2r)$ (well defined by \eqref{e:inclusions}).   
Let $\tilde{\mu}_1\colon X^{V(r)} \to X^{V(2 r)}$ be the $K$-scheme morphism 
induced by the universal property of fibered products by the morphisms  
$f \circ \beta_v  \circ  \pi_{1,v} \colon X^{V(r)} \to X$ for $v\in V(2r)$. 
It follows from construction that  
$\tilde{\mu}_1^{(K)}=\mu_1$. 
Clearly, $\tilde{\mu}_1$ is a homomorphism of algebraic groups (Lemma~\ref{l:homo-alg}). 

\par
Hence $W\coloneqq \tilde{\mu}_1(X^{V(r)})$ is a closed algebraic subgroup of 
$X^{V(2r)}$ by Theorem \ref{t:homo-group}. 
We denote
\[
Z \coloneqq \mu_1 (A^{V(r)})= W(K) \subset A^{V(2r)}. 
\]

Consider the homomorphism $\eta_2 \colon Z \to A^{V(3r)}$ defined 
for each $c \in Z\subset  A^{V(2r)}$ and $v \in V(3r)$ by the formula: 
\[
\eta_2(c)(v) \coloneqq \eta\left(c\vert_{B(v,r)} \circ \psi_{v,r}\right).
\]

\par
Consider the projection $\rho \colon X^{V(r)} \to X^{V(3r)}$ 
then applying \eqref{e:local-inverse} for $\Omega=M=B_S(r)$, we find that 
\[
\eta_2 \circ \mu_1 = \rho^{(K)}.
\] 
It follows that
\begin{equation}
\label{e: Z}
\eta_2(Z) = \eta_2(\mu_1(A^{V(r)}))= \rho^{(K)}(A^{V(r)}) = A^{V(3r)}.
\end{equation}
 
As for $\mu_1$, we verify that $\eta_2$ is induced by a morphism of algebraic varieties 
$\tilde{\eta}_2 \colon W \to X^{V(3r)}$. 
Indeed, let $\pi_{2,v} \colon X^{V(2r)} \to X^{B(v,r)}$ be the projection and 
$\varphi_{2,v} = \pi_{2,v} \circ \iota_W$ where $v\in V(3r)$ 
and $\iota_W \colon W \to X^{V(2r)}$ is the closed immersion.    
Then $\tilde{\eta}_2 \colon W \to X^{V(3 r)}$ is induced by the morphisms  
$h_M \circ \beta_v  \circ  \varphi_{2,v} \colon W \to X$ for $v\in V(2r)$. 
\par
Therefore, we deduce from \eqref{e: Z} that (see, e.g., \cite[Proposition~2.10.(ii)-(iii)]{cscp-alg-goe})
\begin{equation}
\label{e:sofic-dim(Z)}
\dim(Z) \geq \dim(A^{V(3r)})=\vert V(3r)\vert \dim(A)= n \vert V(3r)\vert.
\end{equation}
\par
Now, we define a subset $V' \subset V(3r) $ as in Lemma~\ref{l:sofic-V'}.(ii) so that 
$B(v,r)$ are pairwise disjoint for all $v\in r$ and that $V(3r) \subset \bigcup_{v \in V'}B(v,2r)$.  Let us denote 
$\overline{V'} \coloneqq \coprod_{v \in V'} B(v,r)$.
Since $\overline{V'} \subset V(2r)$ by \eqref{e:inclusions} and $B(v,r)$, $B_S(r)$ are in bijection, we find that 
\begin{equation} 
\label{e:sofic-Vr}
\vert V(2r)\vert = \vert V' \vert \vert B_S(r)\vert + \vert  
V(2r) \setminus \overline{V'}\vert.
\end{equation}
Note that for all  $v \in V(2r)$, we have by the construction of $Z$ and $\mu_1$ an isomorphism: 
\[
Z_{B(v,r)} \to \Gamma_{B_S(r)}=\Gamma_M, \quad c \mapsto c \circ \psi_{v,r}.
\]
As $\Gamma_{B_S(r)} \subsetneq A^{B_S(r)}$ is a proper closed subset 
(by Proposition~\ref{p:set-config-zariski} and \eqref{e:contrary}) and $A^{B_S(r)}$ is irreducible, 
it follows that for all $v \in V'$, we have (cf.~\cite[Proposition~2.5.5]{liu-alg-geom}): 
\begin{equation}
\label{e:z-sofic}
\dim\left(Z_{B(v,r)}\right) = \dim(\Gamma_{B_S(r)}) \leq \dim( A^{B_S(r)}) -1 = n \vert B_S(r)\vert  -1.
\end{equation}
\par

We deduce from the properties of Krull dimension the following estimation:  
\begin{align*}
\dim(Z) & \leq \dim(Z_{\overline{V'}} \times A^{V(2r) \setminus \overline{V'}})
&  (\text{since } Z \subset Z_{\overline{V'}} \times A^{V(2r)}) \\ 
& = \dim\bigl(Z_{\overline{V'}}\bigr) +
\dim\bigl(A^{V(2r) \setminus \overline{V'}}\bigr) 
&  (\text{by, e.g., \cite[Proposition~2.13.(iii)]{cscp-alg-goe}}) \\ 
& \leq   \dim(\prod_{v\in V'} Z_{B(v,r)})  + \dim(A^{V(2r) \setminus \overline{V'}})
&  (\text{since } Z_{\overline{V'}} \subset \prod_{v\in V'} Z_{B(v,r)}) \\ 
& =  \sum_{v\in V'}\dim\bigl(Z_{B(v,r)} \bigr)  + \vert V(2r) \setminus \overline{V'}\vert \dim(A)
&  (\text{cf. \cite[Proposition~2.13.(iii)]{cscp-alg-goe}}) \\
& \leq \vert V' \vert  \bigl( n \vert B_S(r)\vert  - 1 \bigr) + n \vert V(2r) \setminus \overline{V'}\vert 
& (\text{by } \ref{e:z-sofic})\\
& = n \Bigl(\vert V(2r)\vert - \frac{\vert V' \vert}{n}\Bigr) 
&   (\text{by } \ref{e:sofic-Vr}). 
\end{align*}
Combining this inequality with \eqref{e:sofic-dim(Z)}, we obtain the relation:  
\[
\vert V(3r)\vert \leq \vert V(2r)\vert - \frac{\vert V' \vert}{n}.
\]
From this, we deduce that 
\begin{align*}
\vert V \vert \geq \vert V(2r)\vert & \geq \vert V(3r)\vert  +  \frac{\vert V' \vert}{n}
& (\text{since } V\supset V(2r))\\
& \geq  \vert V(3r)\vert  +  \frac{ |V(3r)|}{n \vert B_S(2r)\vert } 
& (\text{by Lemma }\ref{l:sofic-V'})\\
& =  \vert V(3r)\vert \Bigl(1 + \frac{1}{n \vert B_S(2r)\vert }\Bigr) \\
& >  \vert V(3r)\vert (1 - \varepsilon)^{-1} 
& (\text{by } \ref{e:contra-hypothesis-1}).
\end{align*}
It follows that 
\[
 \vert V(3r)\vert < (1 - \varepsilon) \vert V \vert
 \]
which contradicts \eqref{e:sofic-V}.
Therefore, we conclude that 
$\tau$ is surjective and hence bijective. 
\end{proof}
 
\begin{proof}[Proof of Theorem~\ref{t:sofic}]
It results immediately from Theorem \ref{t:sofic-surjunctive} and Theorem \ref{t:inverse-also}. 
\end{proof}

By applying, \emph{mutatis mutandis}, the proof of Theorem \ref{t:sofic-surjunctive}, 
we obtain easily the following property without the restriction that $\car(K)=0$. 

\begin{theorem}
\label{t:inverse-also-near-ring}
Let $G$ be a sofic group. 
Let $X$ be a $K$-algebraic group and $A \coloneqq X(K)$. 
Let $\tau, \sigma \in CA_{alg}(G,X,K)$. 
Suppose that $\sigma\circ \tau= \Id_{A^G}$ and that $\tau \in CA_{algr}(G,X,K)$. 
Then $\sigma \in CA_{algr}(G,X,K)$ and $\tau \circ \sigma= \Id_{A^G}$.  
\end{theorem}

\begin{proof}
To show that $\tau$ is surjective and thus bijective, 
the proof follows the same lines as in the proof of Theorem \ref{t:sofic-surjunctive}. 
The only modification is that since $\sigma \in CA_{alg}(G,X,K)$ and $\sigma \circ \tau = \Id_{A^G}$, 
we can use the local defining maps of $\sigma$ and the corresponding homomorphisms of algebraic groups  
instead of the morphisms $h_\Omega$ given by Proposition \ref{p:left-inverse} where we require that $\car(K)=0$.  
\par
Now as $\tau$ is bijective and $\sigma \circ \tau =\Id_{A^G}$, 
we deduce that $\sigma$ is also bijective and $\tau \circ \sigma =\Id_{A^G}$. 
Since $\tau \in CA_{algr}(G,X,K)$, Proposition \ref{t:inverse-also-near-ring-1} 
implies that $\sigma \in CA_{algr(G,X,K)}$ as well.  
\end{proof}

\section{A version of the Garden of Eden theorem}
\label{s:goe-grp}
In this section, we will prove Theorem~\ref{t:goe-grp} and Corollary~\ref{c:myhill} (cf. Theorem~\ref{t:main-goe}) 
as well as all the relations presented in the diagram in Introduction.  
Let $G$ be a group. 
Let $X$ be a $K$-algebraic variety and $A=X(K)$. 
Let $\tau \colon A^G
\to A^G$ be an algebraic cellular automaton over $(G,X,K)$ with a memory set $M$. 
Let $\Omega \subset G$ be a finite subset, $D\subset A^\Omega$ and $p\in A^{G \setminus \Omega}$. We define 
\[
D_p\coloneqq D \times \{p\} \subset A^G. 
\] 
A subset $\Gamma \subset A^G$ has \emph{finite support} $\Omega$ if $\Gamma=D_p$ for some $\Omega$, $D$, $p$ as above. 
In this case, 
\[
\dim(\Gamma)\coloneqq \dim(D) 
\]  
is well-defined and independent of the choices of $\Omega$, $D$, $p$ such that $\Gamma=D_p$. 
\par
Two notions of weak pre-injectivity are introduced in \cite{cscp-alg-goe}. 
We say that $\tau$ is \emph{$(*)$-pre-injective} if
there do not exist a finite subset $\Omega \subset G$
and a proper closed subset $H \subsetneq A^\Omega $ such that
 \[
 \tau((A^\Omega)_p)=\tau(H_p) \quad \text{ for all } p \in A^{G\setminus \Omega}. 
 \]
The cellular automaton $\tau$ is \emph{$(**)$-pre-injective} if
for every finite subset $\Omega \subset G$, we have 
 \[
 \dim(\tau((A^\Omega)_p))=\dim(A^\Omega) \quad \text{ for some } p \in A^{G\setminus \Omega}. 
 \]
 Note that $\tau((A^\Omega)_p))$ has finite support $\Omega^{+}$. If $X$ is finite, 
 or equivalently, if $\dim(X)=0$, then $(*)$-pre-injectivity is the same as pre-injectivity (cf.~\cite[Example~8.1]{cscp-alg-goe}).    
 \par
For algebraic group cellular automata, 
we shall use the following refined analogous notions of weak pre-injectivity: 

\begin{definition}
\label{d:weak-pre}
Let $G$ be a group. 
Let $(X,e)$ be a $K$-algebraic group and $A=X(K)$. 
Let $\tau \colon A^G
\to A^G$ be an algebraic group cellular automaton over $(G,X,K)$. 
If $D\subset A^\Omega$ for some finite subset $\Omega \subset G$, we write
\[ 
D_e\coloneqq D \times \{e \}^{G \setminus \Omega} \subset A^G. 
\] 
\par
We say that $\tau$ is \emph{$(\sbullet[.9])$-pre-injective} if
there do not exist a finite subset $\Omega \subset G$
 and a subset $H \subsetneq A^\Omega $ closed for the Zariski topology such that
 \[
 \tau((A^\Omega)_e)=\tau(H_e). 
 \]
\par
 We say that $\tau$ is \emph{$( \sbullet[.9] \sbullet[.9]) $-pre-injective} if
for every finite subset $\Omega \subset G$, we have 
 \[
 \dim(\tau((A^\Omega)_e)) = \dim(A^\Omega).   
 \]
\end{definition}

We establish first the following lemma which will be used in Proposition~\ref{p:pre-injectivity-*}. 
\begin{lemma}
\label{l:surj-mor} 
Let $k$ be a field and let $f\colon X \to Y$ be a surjective morphism of irreducible $k$-algebraic varieties. 
Suppose that $\dim(X) > \dim(Y)$. 
Then there exists a proper closed subset $Z\subsetneq X$ such that $f(Z)=Y$.  
 \end{lemma}

\begin{proof}
Let $\eta\in Y$ be the generic point of $Y$ so that $\overline{\{\eta\}}=Y$ (cf.~\cite[Definition~2.4.10]{liu-alg-geom}). 
Restricting $X,Y$ to open affine subschemes if necessary 
to apply the Fiber dimension theorem \cite[Corollary~14.5]{eisenbud-book}, we see that the fiber $f^{-1}(\eta)\subset X$ satisfies
\[
\dim(f^{-1}(\eta))\geq \dim(X)-\dim(Y)\geq 1,
\]
where $\dim(f^{-1}(\eta))$ is the Krull dimension of the subspace $f^{-1}(\eta)$ of $X$.  
Let $\xi$  be the generic point of $X$. 
Then $f(\xi)=\eta$ since $f$ is surjective. 
As $\dim(f^{-1}(\eta))\geq 1$, $f^{-1}(\eta)$ is not a singleton and we can find $x\in f^{-1}(\eta)$ such that $x\neq \xi$. 
Let $F$ be the closure of $x$ in $X$ then $F \subsetneq X$. 
By Chevalley's theorem (cf.~\cite[Th\'eor\`eme~1.8.4]{grothendieck-20-1964}), $f(F)$ is a constructible subset of $Y$. 
Hence, $f(F)$ contains a dense open subset $U$ of its closure $\overline{f(F)}=\overline{\{\eta\}}=Y$. 
Let $V=Y \setminus U \subsetneq Y$ be a proper closed subset. 
Then $T=f^{-1}(V) \subset X$ is a proper closed subset and $f(T)=V$ since $f$ is surjective. 
Let $Z\coloneqq F \cup T \subset X$ then $Z$ is a proper closed subset since $X$ is irreducible and $F, T\subsetneq X$ are proper closed subsets.  
Then by construction, 
\[
f(Z)= f(F) \cup f(T)\supset U \cup V=Y.
\] 
This completes the proof of Lemma \ref{l:surj-mor}. 
\end{proof}

The bottom half of the diagram in Introduction follows from the following proposition and Proposition~\ref{p:preinj-lin}: 

\begin{proposition}
\label{p:pre-injectivity-*}
Let $G$ be a group. 
Let $X$ be a $K$-algebraic group and $A=X(K)$. 
Let $\tau \colon A^G \to A^G$ be an algebraic group cellular automaton over $(G,X,K)$. 
The following hold: 
\begin{enumerate}[\rm (i)]
\item
if $\tau$ is pre-injective then it is $(\sbullet[.9])$-pre-injective and $( \sbullet[.9] \sbullet[.9]) $-pre-injective;
\item
$\tau$ is $(\sbullet[.9])$-pre-injective if and only if it is $(*)$-pre-injective;
\item
$\tau$ is $( \sbullet[.9] \sbullet[.9]) $-pre-injective if and only if it is $(**)$-pre-injective; 
\item
if $X$ is connected then $\tau$ is $( \sbullet[.9] \sbullet[.9]) $-pre-injective if and only if $\tau$ is $(\sbullet[.9])$-pre-injective.  
\end{enumerate}
\end{proposition}

\begin{proof}
Denote by $+$ the group law (not necessarily commutative) of $A^G$ and of the algebraic groups involved. 
Fix a memory set $M$ of $\tau$. 
 \par
For (ii) and (iii), we see that $(\sbullet[.9])$-pre-injectivity (resp. $( \sbullet[.9] \sbullet[.9]) $-pre-injectivity) 
implies trivially $(*)$-pre-injectivity (resp. $(**)$-pre-injectivity). 
Suppose that $\tau$ is not $(\sbullet[.9])$-pre-injective. 
Then there exist a finite subset $\Omega \subset G$ and a proper closed subset $H \subsetneq A^\Omega$ such that 
$\tau((A^\Omega)_e))= \tau(H_e)$.  
Let $p \in A^{G \setminus \Omega}$ and $\tilde{p} \in A^G$ 
be such that $\tilde{p}\vert_\Omega=e^\Omega$ and $\tilde{p}\vert_{A^{G \setminus \Omega}}=p$. 
Then we find that  
\begin{align*}
\tau((A^\Omega)_p) & =\tau((A^\Omega)_e+\tilde{p}) = \tau((A^\Omega)_e) + \tau(\tilde{p}) \\
& =\tau(H_e)+ \tau(\tilde{p})=\tau(H_e+\tilde{p}) \\
& = \tau(H_p).  
\end{align*}
Therefore, $\tau$ is not $(*)$-pre-injective either and this proves (ii). 
Similarly, suppose that $\tau$ is not $( \sbullet[.9] \sbullet[.9]) $-pre-injective. 
Then for some finite subset $\Omega \subset G$, we have 
 $ \dim(\tau((A^\Omega)_e)) <\dim(A^\Omega)$.  
Let $p \in A^{G \setminus \Omega}$ and $\tilde{p} \in A^G$ 
be such that $\tilde{p}\vert_\Omega=e^\Omega$, $\tilde{p}\vert_{A^{G \setminus \Omega}}=p$.
Then we have   
 \begin{align*}
 \dim(\tau((A^\Omega)_p)) & =\dim(\tau((A^\Omega)_e +\tilde{p})) = \dim(\tau((A^\Omega)_e)+\tau(\tilde{p})) \\
 & = \dim(\tau((A^\Omega)_e)) < \dim(A^\Omega).
 \end{align*}
 Thus $\tau$ is not $( \sbullet[.9] \sbullet[.9]) $-pre-injective either and (iii) is proved. 
 \par
 To show (iv), suppose that $X$ is connected thus irreducible. 
 By (ii), (iii) and \cite[Proposition~6.4]{cscp-alg-goe}.(ii), we see that $\tau$ is 
 $(\sbullet[.9])$-pre-injective if it is $( \sbullet[.9] \sbullet[.9]) $-pre-injective. 
 Suppose that $\tau$ is not $( \sbullet[.9] \sbullet[.9]) $-pre-injective. 
 Then there exists a nonempty finite subset $\Omega \subset G$ such that 
 \begin{equation}
 \label{e:**'}
 \dim \left(\tau((A^\Omega)_e)\right)=\dim\left(\tau\left((A^\Omega)_e\right)\vert_{\Omega^+} \right)<\dim(A^\Omega).  
 \end{equation}
Consider the closed immersion homomorphism of $K$-algebraic groups:
\[
\iota \coloneqq (\Id_{X^\Omega}, e^{\Omega^{++} \setminus \Omega}) \colon X^\Omega 
= X^\Omega\times_K \prod_{\Omega^{++}\setminus \Omega} \Spec(K) \to X^{\Omega^{++}}. 
\]
Note that $Z\coloneqq \iota(X^\Omega)$ is homeomorphic to $X^\Omega$ and 
by Theorem \ref{t:homo-group}, it is a $K$-algebraic subgroup of $X^{\Omega^{++}}$.  
Let $j\colon Z \to X^{\Omega^{++}}$ be the corresponding closed immersion and consider 
\[
h\coloneqq  f^+_{\Omega^+} \circ j  \colon Z \to X^{\Omega^+}.
\] 
Then clearly $\sigma \coloneqq h^{(K)} \colon Z(K) \to A^{\Omega^+}$ 
is the restriction of  
$\tau_{\Omega^+}^+$ to $Z(K)=A^\Omega \times \{e\}^{\Omega^{++}\setminus \Omega}$. 
Again by Theorem~\ref{t:homo-group}, $Y =\im(h) \subset X^{\Omega^+}$ is a $K$-algebraic subgroup. 
By construction, we have 
\begin{equation}
\label{e:*1}
Y(K)=\tau^+_{\Omega^+}\left(A^\Omega \times \{e\}^{\Omega{++}}\right) 
= \tau\left((A^\Omega)_e\right)\vert_{\Omega^+}.
\end{equation}

By \cite[Proposition~I.5.2.2]{grothendieck-ega-1}, $h$ factors through a surjective $K$-scheme morphism  
\[
\gamma \colon Z \to Y.
\] 
Then the relation \eqref{e:**'} says that $\dim(Y) < \dim(Z) $. 
Since $X$ is assumed to be irreducible, $X^\Omega$ and thus $Z$ and $Y$ are also irreducible. 
Lemma \ref{l:surj-mor} applied to $\gamma$ and the homeomorphism $Z\simeq X^\Omega$ implies that 
there exists a proper closed subscheme $W \subsetneq X^\Omega$ such that 
\begin{equation}
\label{e:*2}
\gamma \circ \iota (W) = Y.
\end{equation}

Let $H\coloneqq W(K)$, we deduce from \eqref{e:*1} and \eqref{e:*2} that $\tau(H_e)=\tau((A^\Omega)_e)$.  
Therefore, $\tau$ is not $(\sbullet[.9])$-pre-injective and this proves (iv). 
\par
As pre-injectivity implies $(*)$- and $(**)$-pre-injectivity (cf.~\cite[Proposition~5.3.(i)]{cscp-alg-goe}), 
the assertion (i) follows from (ii) and (iii). 
\end{proof}

\begin{proposition}
\label{p:sur-dim}
Let $G$ be an amenable group with a F\o lnet net $\FF=(F_i)_{i \in I}$.  
Let $X$ be a connected $K$-algebraic group and $A \coloneqq X(K)$.
Let $\tau \colon A^G \to A^G$ be an algebraic group cellular automaton 
over $(G,X,K)$.  
Suppose that  $\mdim_\FF(\tau(A^G)) =\dim(A)$.
Then $\tau$ is surjective.
\end{proposition}

\begin{proof}
We apply \cite[Lemma~5.3]{cscp-alg-goe} to the $G$-invariant subset $\Gamma \coloneqq \tau(A^G)$. 
Let us check the three conditions of \cite[Lemma~5.3]{cscp-alg-goe}. 
The condition (D1) is satisfied by the closed image property of $\tau$ (cf.~Theorem~\ref{t:closed-image}). 
The condition (D2), i.e., $\Gamma_\Omega$ is closed in $A^\Omega$ 
for the Zariski topology for every finite subset $\Omega \subset G$, is satisfied by 
Proposition~\ref{p:set-config-zariski}. 
Finally, the condition (D3) is exactly the hypothesis $\mdim_\FF(\Gamma) =\dim(A)$. 
Hence, \cite[Lemma~5.3]{cscp-alg-goe} implies that $\Gamma=A^G$ and thus $\tau$ is surjective.  
\end{proof}

\begin{proposition}
\label{p:**-implies-mdim}
Let $G$ be an amenable group with a F\o lner net $\FF=(F_i)_{i \in I}$. 
Let $X$ be a $K$-algebraic group and $A \coloneqq X(K)$. 
Let $\tau \colon A^G \to A^G$ be an algebraic group cellular automaton over $(G,X,K)$. 
The following hold: 
\begin{enumerate} [\rm (i)]
\item
if $\tau$ is $( \sbullet[.9] \sbullet[.9]) $-pre-injective then $\mdim_\FF(\tau(A^G)) = \dim(X)$;
\item
if $X$ is connected and $\mdim_\FF(\tau(A^G)) = \dim(X)$ then $\tau$ is $(\sbullet[.9])$-pre-injective. 
\end{enumerate}
\end{proposition}

\begin{proof}
The assertion (i) is a consequence of \cite[Proposition~6.5]{cscp-alg-goe} 
and the equivalence between $(**)$- and $( \sbullet[.9] \sbullet[.9]) $-pre-injectivity 
for algebraic group cellular automata (Proposition~\ref{p:pre-injectivity-*}.(iii)). 
Similarly, the assertion (ii) follows from \cite[Proposition~6.6]{cscp-alg-goe} 
and the equivalence between $(*)$- and $(\sbullet[.9])$-pre-injectivity 
in our case (Proposition~\ref{p:pre-injectivity-*}.(ii)). 
Remark that the algebraic group $X$ is irreducible if it is connected.  
\end{proof}

Theorem~\ref{t:goe-grp} and Corollary~\ref{c:myhill} are consequences of the following. 

\begin{theorem}
\label{t:main-goe}
Let $G$ be an amenable group. 
Let $X$ be a $K$-algebraic group and $A \coloneqq X(K)$. 
Let $\tau \colon A^G \to A^G$ be an algebraic group cellular automaton over $(G,X,K)$. 
Then the following are equivalent: 
\begin{enumerate} [\rm (a)] 
\item
$\tau$ is surjective; 
\item
$\tau$ is $(\sbullet[.9])$-pre-injective; 
\item
$\mdim_\FF(\tau(A^G)) =\dim(X)$ for some (or equivalently any) F\o lner net $\FF$ of the group $G$.  
\end{enumerate}
In particular, $\tau$ is surjective whenever it is pre-injective. 
\end{theorem}

\begin{proof}
It is clear by definition of algebraic mean dimension that (a) $\Rightarrow$ (c). 
The converse implication follows from Proposition~\ref{p:sur-dim}. 
We deduce that (a) $\Leftrightarrow$ (c).
Now (c) $\Rightarrow$ (b) follows from Proposition~\ref{p:**-implies-mdim}.(ii), 
and (b) $\Rightarrow$ (c) from Proposition~\ref{p:**-implies-mdim}.(i) and Proposition~\ref{p:pre-injectivity-*}.(iv). 
Hence, the conditions (a), (b), (c) are equivalent. 
\par
By Proposition~\ref{p:pre-injectivity-*}.(i),  pre-injectivity implies $(\sbullet[.9])$-pre-injectivity. 
Together with the implication (b) $\Rightarrow$ (a), this proves the last assertion. 
\end{proof}

Now suppose that $G$ is an amenable group and $X$ is a connected $K$-algebraic group. 
Let $f \colon X^M \to X$ be a homomorphism of $K$-algebraic groups 
where $M \subset G$ is a finite subset. 
If $L/K$ is a field extension, denote $X_L\coloneqq X \times_K \Spec(L)$ 
the $L$-algebraic group obtained by the base change $\Spec(L) \to \Spec(K)$. 
Observe that $X_L(L)=X(L)$. 
We define $\tau^{(L)}\colon X(L)^G \to X(L)^G$ 
to be the algebraic group cellular automaton over $(G,X_L, L)$ with memory set $M$ 
and associated local defining map $f^{(L)}$. 
By using Theorem~\ref{t:main-goe}, the proof of the following result is the same, 
\emph{mutatis mutandis}, as the proof of \cite[Theorem~7.2]{cscp-alg-goe}. 

\begin{proposition}
\label{t:base-surj}
With the above notation, suppose in addition that $L$ is an algebraically closed field. 
Then $\tau^{(K)}$ is surjective if and only if $\tau^{(L)}$ is surjective.  
\end{proposition}

The next proposition shows that for linear cellular automata, 
all the weak notions of pre-injectivity that we have introduced turn out to be equivalent to pre-injectivity. 

\begin{proposition}
\label{p:preinj-lin}
Let $G$ be a group. 
Let $X=\A^n$ ($n\in \N$) be the $n$-dimensional affine space over $K$ and $A\coloneqq X(K)=K^n$. 
Let $\tau \colon A^G \to A^G$ be a $K$-linear cellular automaton. 
Then the properties of T-pre-injectivity for $\tau$ where  
\[
T \in \{(*), (**), (\sbullet[.9]), ( \sbullet[.9] \sbullet[.9]) \}
\] 
are all equivalent and they are equivalent to pre-injectivity of $\tau$. 
\end{proposition}

\begin{proof} 
We first remark that $\tau$ is an algebraic group cellular automaton over $(G, X ,K)$ 
where $X$ is considered as the additive algebraic group $\G_a^n$ which is connected. 
Hence, it follows immediately from Proposition~\ref{p:pre-injectivity-*}.(ii)-(iv) that 
the properties of T-pre-injectivity, for $T\in \{(*), (**), (\sbullet[.9]), ( \sbullet[.9] \sbullet[.9]) \}$, are all equivalent. 
\par
Suppose now that $\tau$ is not pre-injective. 
Then the argument given in \cite[Example~8.6]{cscp-alg-goe} shows that $\tau$ is not $(*)$-pre-injective.  
As pre-injectivity implies $(\sbullet[.9])$-pre-injectivity 
by Proposition~\ref{p:pre-injectivity-*}.(i), the proof of Proposition \ref{p:preinj-lin} is completed. 
\end{proof}

\section{Kaplansky's direct finiteness conjecture}
\label{s:kap-finite}
Let $G$ be a group. 
Let $(X,e)$ be an algebraic group over an algebraically closed field $K$. 
Denote by $A=X(K)$ the set of $K$-points of $X$. 
Denote by $R\coloneqq \End_K(X)$ the set of endomorphisms of $X$. 
Then $R$ is naturally a monoid with the law induced by the composition of endomorphisms. 
The identity map is the unit of this monoid. 
\par
If $X$ is a commutative algebraic group, then $R$ is an abelian group for pointwise addition 
and this makes $R$ a ring together with the monoid structure described above.
For example, 
if $X=\G_a$ is the additive group over $K=\overline{\F_p}$ 
where $p$ is a prime number,  then $R=K[T; F]$ is the skewed polynomial ring 
associated to the Frobenius morphism of $K$ 
with the twisted multiplication $T^na\coloneqq F^n(a)T^n=a^{p^n}T^n$. 

\subsection{Group rings and algebraic group cellular automata} 
We assume that $X$ is a commutative algebraic group over $K$ so that $R$ is a ring. 
We can use the symbol $\sum$ everywhere for the (commutative) additions on $X$ or $A$. 
We begin with an immediate observation. 

\begin{lemma}
$A^G$ is canonically an $R[G]$-module compatible with the $G$-shift 
and the natural action of $R$ on $A^G$ as follows: 
for all $\alpha= \sum_{h \in G} \alpha(h)h \in R[G]$ and $x\in A^G$, 
\begin{equation}
\label{e:R[G]-module}
(\alpha x) (g) \coloneqq \sum_{h\in G} \alpha(h) x(h^{-1} g) \quad \text{for all } g \in G.  
\end{equation}
\end{lemma}

\begin{proof}
The proof is straightforward and is omitted. 
\end{proof}

We define $A[G] \subset A^G$ to be the subset consisting of all configurations with finite support. 
For every $z \in A$, we define an element $c_z \in A[G]$ support on $1_G$ by setting
\[   
c_z(g) \coloneqq 
     \begin{cases}
       z &\quad\text{if } g=1_G;  \\
       e & \quad\text{otherwise.} \\ 
     \end{cases}
\]
With each $\alpha \in R[G]$, we associate a map $\tau_\alpha \colon A^G \to A^G$ as follows: 
\begin{equation}
\tau_\alpha (x) (g) \coloneqq \sum_{h\in G} \alpha(h) (x(gh)), \text{ for all } x \in A^G, g \in G.   
\end{equation}

\begin{lemma}
The map $\tau_\alpha \colon A^G \to A^G$ is an algebraic group cellular automaton over $(G,X,K)$. 
Moreover, the support of $\alpha$ is the minimal memory set of $\tau_\alpha$. 
\end{lemma}

\begin{proof}
Let $M\subset G$ denote the support of $\alpha$. 
We define a homomorphism of algebraic groups $f \colon X^M \to X$ by setting 
\[
f= m_{X^M} \circ \prod_{h\in M} \alpha(h) , 
\]
where $m_{X^M} \colon X^M \to X$ is the sum morphism induced by the group law of $X$. 
Note that $m_{X^M}$ is well-defined since $X$ is a commutative algebraic group. 
Let $\mu\coloneqq f^{(K)} \colon A^M \to A$ be the induced map. 
Then it is immediate that 
\begin{equation}
\label{e:mu}
\mu(z) \coloneqq \sum_{h\in M} \alpha(h) (z(h)), \text{ for all } z \in A^M. 
\end{equation}

For every $x \in A^G$ and $g\in G$, 
\begin{align*}
\tau_\alpha(x)(g) & \coloneqq \sum_{h\in G}  \alpha(h) (x(gh)) = \sum_{h\in G}  \alpha(h) (g^{-1}x(h)) \\
& = \sum_{h\in M}  \alpha(h) (g^{-1} x(h)) = \mu((g^{-1}x)\vert_M) & \text{ (by \eqref{e:mu})}
\end{align*}
 
It follows that $\tau_\alpha \colon A^G \to A^G$ is an algebraic group cellular automaton over $(G,X,K)$.
\par
Suppose that $S$ is the minimal memory set of $\tau$. 
Then $S \subset M$ since $M$ is a memory set of $\tau$. 
Let $\nu \colon A^S \to A $ be the corresponding local defining map to $S$, 
which is automatically a homomorphism of groups. 
Then for all $z\in A$ and $g\in G$, we have 
\[
\nu((gc_z)\vert_S)=  \tau_\alpha(c_z)(g^{-1}) = \sum_{h\in G}\alpha(h)(c_z(g^{-1}h))= \alpha(g)(z). 
\]
For every $g\notin S$, the support $gc_z$ is $g$ and thus $\nu((gc_z)\vert_S)=\nu(e^S)=e$. 
Consequently, we have $\alpha(g)(z)=e$ for all $z\in A$ and thus $\alpha(g)=0$ for all $g\notin S$. 
This implies that $M \subset S$ and we conclude that $M=S$ is the minimal memory set of $\tau$. 
\end{proof}

The following proposition proves the first part of Theorem \ref{t:kap-direct-finite}. 

\begin{proposition}
\label{p:iso-ca-gr}
The map $\Psi \colon R[G] \to CA_{algr}(G,X,K)$ given by $\Psi(\alpha) = \tau_\alpha$ 
for all $\alpha \in R[G]$ is an $R$-linear isomorphism of rings. 
\end{proposition}

\begin{proof}
We first show that $\Psi$ is injective. 
Indeed, if $\tau_\alpha=0$ for some $\alpha \in R[G]$ then for every $z \in A$ and $g\in G$, we find that 
\[
e= \tau_\alpha(c_z)(g^{-1}) = \sum_{h\in G}\alpha(h)(c_z(g^{-1}h))= \alpha(g)(z). 
\]
Therefore, as $z \in A$ is arbitrary and $A$ is dense in $X$, we obtain $\alpha(g)=0$ for all $g \in G$. 
It follows that $\alpha=0$ and thus $\Psi$ is injective. 
\par
Now to show that $\tau$ is surjective, let $\tau \in CA_{algr}(G,X,K)$ with memory set $M$ 
and the associated local defining map $\mu \colon A^M \to A$ induced 
by a homomorphism of algebraic groups $f\colon X^M \to X$. 
For every $h\in M$, consider the endomorphism $\gamma_h \colon X^{\{h\}} \to X$ 
given by the composition of $f$ with the embedding $(\Id_{X^{\{h\}}} \times e^{M \setminus \{h\}})\colon X^{\{h\}} \to X^M$. 
Setting $\alpha_h\coloneqq \gamma_h^{(K)} \colon A \to A$, it follows immediately that 
\[
\mu(z) = \sum_{h \in M} \alpha_h(z(h)), \text{ for all } z \in A^M. 
\]
Define $\alpha\coloneqq \sum_{h\in M} \gamma_h h \in R[G]$. 
Then we find that 
\[
\tau(x)(g)= \mu ((g^{-1}x)\vert_M)= \sum_{h\in M} \alpha_h (g^{-1}x(h))= \sum_{h \in M} \alpha_h (x(gh)) \eqqcolon \tau_\alpha(x)(g), 
\]
for all $x\in A^G$ and $g\in G$. 
We deduce that $\tau= \tau_\alpha$ as desired. 
\par
Let $u\in R$ and $\alpha, \beta \in R[G]$, we deduce easily from the commutativity of $R$ that 
\[
\tau_{\alpha + \beta}= \tau_\alpha +\tau_\beta, \quad \text{and} \quad \tau_{u\alpha} = u\tau_\alpha. 
\] 
In particular, we have $\Psi(\alpha + \beta)= \Psi(\alpha) + \Psi(\beta)$. 
We verify now that $\Psi$ is a ring homomorphism. 
Let $x\in A^G$ and denote $y= \tau_\beta(x) \in A^G$. 
Then for all $g \in G$, we find  
\begin{align*}
\tau_\alpha (\tau_\beta(x)))(g) & = \tau_\alpha(y)(g) \coloneqq \sum_{h \in G} \alpha(h) (y(gh)))\\
&= \sum_{h \in H} \alpha(h) \left( \sum_{k \in G} \beta(k) (x(ghk)) \right) \\
&= \sum_{h \in G} \alpha(h) \left(  \sum_{t=hk \in G} \beta(h^{-1}t) (x(gt))     \right) \\
&= \sum_{t\in G} \left(  \sum_{h \in G} \alpha(h) \beta(h^{-1}t)  \right) (x(gt)) \\
&= \sum_{t \in G} (\alpha \beta) (t) (x(gt)) = \tau_{\alpha \beta}(x)(g).
\end{align*}
Therefore, $\tau_\alpha \circ \tau_\beta = \tau_{\alpha \beta}$ and thus $\Psi(\alpha \beta)= \Psi(\alpha) \Psi(\beta)$ for all $\alpha, \beta \in R[G]$. 
As clearly $\Psi(1_{R[G]})=\Id_{A^G}$, we conclude that $\Psi$ is an isomorphism of rings. 
\end{proof}

\subsection{Application to Kaplansky's direct finiteness conjecture}
In this section, we will complete the proof of Theorem \ref{t:kap-direct-finite}. 
We begin with two easy lemmata. 

\begin{lemma}
\label{l:mat-1}
Let $R$ be a ring and let $G$ be a group. 
Then, for every $n \in \N$, we have a canonical ring isomorphism 
\begin{align*}
T \colon \Mat_n(R) [G] & \to \Mat_n(R[G]) \\
\sum_{g\in G} A(g) g & \mapsto \left( \sum_{g\in G} A(g)_{ij} g \right)_{1 \leq i,j \leq n}.   
\end{align*} 

\end{lemma}

\begin{proof}
Note that $T(\Id_{n} 1_G) =\Id_n$. 
It is clear that $T$ is an isomorphism of groups. 
Hence, we only need to check that for all $g,h \in G$ and $A, B \in \Mat_n (R)$, 
we have $T(Ag.Bh)=T(Ag)T(Bh)$. 
But this is true by the definition of matrix multiplication. 
\end{proof}

\begin{lemma}
\label{l:mat}
For any nonempty finite set $M$  of cardinal $m\in \N$ and a commutative $K$-algebraic group $X$, 
we have a canonical isomorphism of rings: 
\[
\End_K(X^M)= \M_m(\End(X)). 
\]
\end{lemma}

\begin{proof}
We can suppose that $M=\{1,\dots,m\}$. 
By the universal property of fibered products and the commutativity of $X$, 
an endomorphism of algebraic groups $\alpha \colon X^M \to X^M$ 
is the data of $m^2$ endomorphisms of algebraic groups 
$\alpha_{ij} \colon X^{\{i\}} \to X^{\{j\}}$ for $1 \leq i, j \leq m$. 
\par 
Define a bijective homomorphism of additive groups 
$\Pi \colon \End_K(X^M) \to \M_m(\End(X))$ by setting for every 
$\alpha=(\alpha_{ij})_{1 \leq i, j \leq m} \in \End(X^M)$ the tranpose matrix
\[
\Pi(\alpha)= \transp{\alpha} \in \M_m(\End(X)).
\] 
For all $\alpha, \beta \in \End_K(X^M)$, 
let $\alpha \circ \beta=(\gamma_{ij})_{1 \leq i, j \leq m} \in \End_K(X^M)$. 
Every $\gamma_{ij} \colon X^{\{i\}} \to  X^{\{j\}}$ is then just the sum of the endomorphisms 
$\alpha_{kj} \beta_{ik} \colon X^{\{i\}}\to X^{\{k\}} \to X^{\{j\}}$ for $1\leq k \leq m$:
\[
\gamma_{ij}= \sum_{k=1}^m \alpha_{kj} \beta_{ik}. 
\]
It follows that 
$\Pi(\alpha \circ \beta)=\Pi(\alpha)  \Pi(\beta)$ and $\Pi$ is indeed an isomorphism of rings. 
\end{proof}

\begin{corollary}
\label{c:kap-direct-finite}
Let $G$ be a group and let $X$ be a commutative $K$-algebraic group. 
Let $R= \End_K(X)$ be the endomorphism ring of $X$. 
Then the group ring $R[G]$ is stably finite if one of the following holds.
\begin{enumerate} [\rm (i)]
\item
$G$ is sofic and $X$ is connected;  
\item
$G$ is a locally residually finite group.  
\end{enumerate}
\end{corollary}

\begin{proof}
One observes first that the direct finiteness of $R[G]$ follows immediately from 
Proposition \ref{p:iso-ca-gr} and Theorem \ref{t:inverse-also-near-ring}. 
To show that $R[G]$ is stably finite, we have to verifiy that for every $n\in \N$, the ring $\Mat_n(R[G])$ is directly finite. 
Consider the $K$-algebraic group $Y=X^n$. 
Then $Y$ is connected if $X$ is connected since $K$ is algebraically closed. 
By Lemma \ref{l:mat}, we have $\End_K(Y)\simeq \Mat_n(R)$. 
It follows that the ring $\Mat_n(R)[G]$ is directly finite by the above observation. 
This finishes the proof since we have $ \Mat_n(R) [G] \simeq \Mat_n(R[G])$ by Lemma \ref{l:mat-1}. 
\end{proof}

\section{The near ring $R(K,G)$ and generalized Kaplansky's conjectures}
\label{s:near-ring}
\subsection{Definition}

From now on, let $K$ be a ring and let $G$ be a group. 
We describe in this section a natural near ring $R(K,G)$ associated to $K$ an $G$. 
\par
As a group, $( R(K,G), +)$ is identified with the additive group of the polynomial ring 
with coefficients in $K$ in variables indexed by $G$. 
Thus, $(R(K,G),+)=(K[X_g: g\in G],+)$ as groups with the zero polynomial $0$ as neutral element. 
We consider the following set of non-negative integral functions on $G$ with finite support: 
\[
\N^G_c \coloneqq \{ f\colon G \to \N : f(g)=0 \textit{ for all but finitely many } g\in G \}, 
\]
which can be viewed as the semi-ring $\N[G]$.  
We denote the (finite) support of $u\in \N^G_c$ by 
\[
\supp(u)\coloneqq \{ g\in G: u(g)\neq 0 \} \subset G
\]
and define the associated monomial 
\[
\overline{X}^u \coloneqq \prod_{g \in G} X_g^{u(g)}
\]
with the usual identity $X_g^0 =1$. 
Observe that $R(K,G)$ is a free $K$-module with a basis given by the monomials $\overline{X}^u$ for all $u\in \N^G_c$. 
For every $\alpha \in R(K, G)$, we denote also 
\[
\supp(\alpha) \coloneqq \cup_{u \in \N^G_c, \,\alpha(u)\neq 0} \supp(u). 
\] 
There are natural shift actions of the group $G$ on $\N^G_c$ and $R(K,G)$ as follows.  
For all  $g\in G$, $u \in N^G_c$, and $\gamma=\sum_{w \in \N^G_c} \gamma(w) \overline{X}^{w} \in R(K,G)$, 
we define the integral finitely supported function 
\begin{align}
gu \colon G \to \N, \quad (gu)(h)\coloneqq u(g^{-1}h), \quad \text{for every } h \in G, 
\end{align}
and the polynomial 
\begin{align}
 g\gamma \coloneqq \sum_{w \in \N^G_c} \gamma(w) \overline{X}^{gw} \in R(K,G). 
\end{align}

If $u, v \in \N^G_c$ then we define the function $u\star v \in \N^G_c$ by the following convolution formula: 
\begin{align}
\label{e:convolution}
(u \star v) (g) \coloneqq \sum_{h \in G} u(h) v(h^{-1}g), \quad \text{for all } g\in G. 
\end{align}

\begin{lemma}
\label{l:conv-prd}
For all $u, v \in \N^G_c$ and $g\in G$, we have the following properties 
\begin{enumerate} [\rm (i)]
\item
$u\star v=\sum_{h\in G} u(h) (hv)$; 
\item
$\supp(u\star v)=\supp(u)\supp(v)$; 
\item
$\supp(gu)=g\supp(u)$. 
\end{enumerate}
\end{lemma} 

\begin{proof}
The points (i) and (iii) are trivial. 
For (ii), remark that the functions $u, v$ have values in $\N$. 
If $u(h), v(h^{-1}g) > 0$ for some $h,g \in G$ then $(u\star v) (g) >0$. 
It follows from \eqref{e:convolution} that $(u\star v ) (g) \geq u(h) v(h^{-1}g) >0$ 
and thus $\supp(u\star v) \supset \supp(u)\supp(v)$. 
Suppose now that $g \in \supp(u\star v)$. 
Then \eqref{e:convolution} implies that there exist $h \in G$ such that $u(h), v(h^{-1}g)>0$. 
This proves the other inclusion and thus $\supp(u\star v) = \supp(u)\supp(v)$. 
\end{proof}

Suppose that 
$\alpha=\sum_{u \in \N^G_c} \alpha(u) \overline{X}^{u}$ and 
$\beta=\sum_{v \in \N^G_c} \beta(v) \overline{X}^{v}$ are two elements in $R(K,G)$. 
Note that $\alpha(u)=0$ and $\beta(v)=0$ for all but finitely many $u,v \in \N^G_c$. 
The multiplication law also denoted by $\star$ on $R(K,G)$ is given as follows. 
First, for each $u\in \N^G_c$, we define 
\begin{align}
\overline{X}^u \star \beta \coloneqq \prod_{g \in G} \left( g\beta \right)^{u(g)}, 
\end{align}
and then extend by $K$-linearity to obtain $\star \colon R(K,G) \times R(K,G) \to R(K,G)$ given by  
\begin{align*}
\alpha \star \beta & \coloneqq \sum_{u \in \N^G_c} \alpha(u) \overline{X}^u\star\beta \\
& = \sum_{u \in \N^G_c} \alpha(u) \prod_{g \in G} (g\beta)^{u(g)}\\
&=  \sum_{u \in \N^G_c} \alpha(u) \prod_{g \in G} \left(\sum_{v \in  \N^G_c}\beta(v) \overline{X}^{gv} \right)^{u(g)}. \numberthis \label{e:def-nr}
\end{align*}

\begin{remark} 
\label{r:1*0=1}
It follows immediately from the above definition that for all $\alpha \in R(K,G)$,  
\begin{enumerate} [\rm (i)]
\item
$\alpha \star 1= \sum_{u \in \N^G_c} \alpha(u)$; 
\item
$1 \star \alpha= 1$. 
\end{enumerate}

\end{remark}

\begin{example}
Let $g,g',h,h' \in G$ and 
$\alpha = X_g^3X_{h}$ + 1, $\beta = X_{g'}^2 - X_{h'}^2$. 
Then   
\[
\alpha \star \beta = (X_{gg'}^2 - X_{gh'}^2)^3 (X_{hg'}^2 - X_{hh'}^2)+1, 
\]
and 
\[
\beta \star \alpha = (X_{g'g}^3X_{g'h}+1)^2 - (X_{h'g}^3X_{h'h}+1)^2. 
\]
\end{example}

We have the following algebraic structure property of $R(K,G)$. 

\begin{proposition}
Let $K$ be a ring with unit and let $G$ be a group. 
Then $(R(K,G), +, \star )$ is a left near ring, i.e., 
\begin{enumerate}[\rm (i)] 
\item
$(R(K,G), +)\simeq (K[X_g: g \in G],+)$ is a commutative group with neutral element $0$; 
\item
$(R(K,G), \star )$ is a monoid with identity element $X_{1_G}$;
\item
for all $\alpha, \beta, \gamma \in R(K,G)$, we have 
\[
(\alpha +\beta) \star \gamma = \alpha \star \gamma + \beta \star \gamma.  
\]
\end{enumerate}
If the near ring $R(K,G)$ is commutative
then $K$ and $G$ are  commutative as well.  
\end{proposition}

\begin{proof}
The first three assertions are straightforward from the above definition of $R(K,G)$. 
To verify without explicitly writing out a simple but tedious calculation that the map $\star$ is associative, 
we can use the isomorphism $R(K,G) \simeq CA_{alg}(G,\A^1, K)$ in Theorem \ref{t:iso-nr-ca} 
when $K$ is an infinite field and the obvious corresponding 
associativity property of the composition rule in $CA_{alg}(G,\A^1, K)$.  
For the last assertion, suppose that $R(K,G)$ is commutative. 
Then we have $X_g \star X_h = X_h \star X_g$ for all $g,h \in G$. 
It follows that $X_{gh}=X_{hg}$, i.e., $gh=hg$ for all $g,h \in G$, i.e.,  $G$ is commutative. 
Similarly, $aX_{1_G} \star b X_{1_G} = bX_{1_G} \star aX_{1_G}$ implies that $ab=ba$ for all $a,b \in K$. 
It follows that $K$ is commutative. 
\end{proof}

Note that a near ring may not satisfy the relation $\alpha \star 0=0$. 
In fact, we always have in $R(K,G)$ that $1\star 0=1$ (Remark \ref{r:1*0=1}).  

If $\car(K)=p >0$, let $K[t; F]$ be the twisted polynomial ring whose multiplication is twisted by the Frobenius map: 
\[
Ta\coloneqq a^p T.
\] 
Note that in this case, $K[t;F]$ is exactly the endomorphism ring of the additive group $\G_a=\Spec K[X]$. 
We have the following natural embeddings:  

\begin{proposition}
\label{p:embed-ring}
Let $K$ be a ring and let $G$ be a group. 
We have a canonical embedding of near rings 
\[
K[G] \to R(K,G).
\] 
If $K$ is a ring of characteristic $p>0$, then the above embedding factors 
as a composition of canonical embeddings of near rings
\[
K[G] \to K[t;F][G] \to R(K,G).  
\]
\end{proposition} 

\begin{proof}
We consider the following map 
\begin{align*}
\iota \colon  K[G] &\to R(K,G) \\
\sum_{g \in G} \alpha(g)g & \mapsto \sum_{g \in G} \alpha(g) X_g . 
\end{align*}
Clearly $\iota$ is well-defined and is a homomorphism of commutative groups. 
Let $\alpha= \sum_{g \in G} \alpha(g)g$ and 
$\beta= \sum_{h \in G} \beta(h)h$ be elements in $K[G]$. 
We check now that $\iota (\alpha \beta) = \iota(\alpha) \star \iota(\beta)$. 
Indeed, 
\begin{align*}
\iota (\alpha \beta) &= \iota\left(\sum_{g,h \in G} 
\alpha(g)\beta(h)gh\right) \\
& \coloneqq \sum_{g,h \in G} \alpha(g) \beta(h) X_{gh}, 
\end{align*}
and on the other hand 
\begin{align*}
\iota(\alpha) \star \iota(\beta) & = \left(\sum_{g \in G} \alpha(g) X_g\right)\star \left(\sum_{h \in G} \beta(h)X_h\right) \\
& \coloneqq \sum_{g\in G} \alpha(g)\left(\sum_{h \in G} \beta(h)X_{gh}\right)\\
&= \sum_{g,h \in G} \alpha(g) \beta(h) X_{gh}. 
\end{align*}

Suppose now that $\iota(\alpha)= \iota
(\beta)$. 
Then we have $\sum_{g \in G} \alpha(g) X_g = \sum_{g \in G} \beta(g) X_g$  
thus $\alpha(g) =\beta(g)$ for all $g\in G$, i.e., $\alpha=\beta$. 
It follows that $\iota$ is an injective homomorphism of near rings. 
\par
Now suppose that $\car(K)=p >0$. For every $g\in G$, $A= \sum_{k\geq 0} a_k t^k \in K[t; F]$, 
denote 
\[
P(A;g)\coloneqq \sum_{k\geq 0} a_k X_g^{p^k} \in R(K,G). 
\]
It follows that $P(A;g)=0$ if and only if $A=0$. 
For $A,B \in F(K)$ and $g \in G$, we have  
\begin{equation}
\label{e:P-add}
P(A+B;g)=P(A;g)+P(B;g). 
\end{equation}
Now consider the map 
\begin{align*}
j \colon K[t; F] [G] & \to R(K,G)\\
\sum_{g \in G} \alpha(g) g &\mapsto \sum_{g\in G} P(\alpha(g); g). 
\end{align*} 
From \eqref{e:P-add}, we find that $j$ is a homomorphism of commutative groups. 
\par
Let $\alpha= \sum_{g \in G} \alpha(g)g$ and 
$\beta= \sum_{h \in G} \beta(h)h$ be elements in $K[t ;F][G]$. 
We check now that $j (\alpha \beta) = j(\alpha) \star j (\beta)$. 
Since $R(K,G)$ is a left near ring and by linearity, 
we can suppose that $\alpha = A g$ for some $g\in G$ and 
$A= \sum_{k\geq 0} a_k t^k \in K[t; F]$. 
We write $\beta(h) = \sum_{r \geq 0} b_r(h) t^r$ for each $h\in G$. 
Let $C(h)=\sum_{s \geq 0}c_s(h) t^s = A \beta(h)$ so that $c_s(h)=\sum_{0\leq k \leq s } a_k b_{s-k}(h)^{p^k}$ for $s\in \N$. 
We see that 
\begin{align*}
j (\alpha \beta) = j\left(\sum_{h\in G}C(h) gh \right)= \sum_{h\in G} \sum_{s\geq 0}c_s(h)X_{gh}^{p^s}. 
\end{align*}
On the other hand, note that since we are in characteristic $p >0$, we find that 
\begin{align*}
j (\alpha) \star j(\beta) 
&= \left( \sum_{k\geq 0} a_k X_{g}^{p^k} \right) \star 
\left( \sum_{h\in G} \sum_{r\geq 0} b_r(h) X_{h}^{p^r} \right) \\
& \coloneqq \sum_{k\geq 0} a_k \left( \sum_{h \in G} \sum_{r\geq 0} b_r(h) X_{gh}^{p^r} \right)^{p^k}\\
&= \sum_{k\geq 0} a_k \sum_{h \in G} \sum_{r\geq 0} b_r(h)^{p^k} X_{gh}^{p^{r+k}}\\
& = \sum_{h\in G} \sum_{s \geq 0} c_sX_{gh}^{p^s}. 
\end{align*}

It follows that $j (\alpha \beta) = j(\alpha) \star j (\beta)$. 
Now suppose that $j(\alpha) =0$ for some $\alpha =\sum_{g\in G} \alpha(g) g \in K[t;F] [G]$. 
Then for all $g \in G$, $P(\alpha(g); g)=0$. 
It follows that $\alpha(g)=0$ for all $g\in G$, i.e., $\alpha=0$. 
We conclude that $j$ is an injective homomorphism of near rings. 
\par
Finally, we see from the constructions that the map $\iota \colon K[G] \to R(K,G)$ 
factors as a composition of the obvious  inclusion of rings
$\theta \colon K[G] \to K[t;F][G]$ followed 
by the embedding $j \colon K[t;F][G] \to R(K,G)$. 
\end{proof}

\subsection{The isomorphism $R(K,G)\simeq CA_{alg}(G, \A^1, K)$} 
In this section, we suppose that $K$ is a commutative ring. 
We will establish in Theorem \ref{t:iso-nr-ca} a point-wise description 
of the near ring $R(K,G)$ in terms of $CA_{alg}(G, \A^1, K)$ 
when the ring $K$ has a certain nice property. 
To each element $\alpha=\sum_{u \in \N^G_c} \alpha(u) \overline{X}^{u}$ in $R(K,G)$, 
we associate a map $\sigma_\alpha \colon A^G \to A^G$ where $A=\A^1(K)=K$ as follows: 
\begin{align}
\label{e:alg-nr} 
\sigma_\alpha (x) (g) 
 \coloneqq \sum_{u \in \N^G_c} \alpha(u) x^{gu} 
 = \sum_{u \in \N^G_c} \alpha(g^{-1}u) x^{u},  
 \quad \text{ for all } x \in A^G, g \in G, 
\end{align}
where $y^{v}\coloneqq  \prod_{h\in G}y(h)^{v(h)}$ for all $y\in A^G$ and $v \in \N_c^G$.  

\begin{lemma}
\label{l:map-phi}
The map $\sigma_\alpha \colon A^G \to A^G$ is an algebraic cellular automaton over $(G,\A^1,K)$. 
Moreover, $\supp(\alpha)$ is the minimal memory set of $\tau_\alpha$. 
In fact, every $\sigma \in CA_{alg}(G, \A^1, K)$ is of the form $\sigma=\sigma_\beta$ for some $\beta \in R(K,G)$. 
\end{lemma}

\begin{proof} 
Since $\alpha(u)=0$ for almost all $u\in \N_c^G$, 
it follows that $M\coloneqq \supp(\alpha) \subset G$ is finite. 
Consider the morphism of affine $K$-schemes $f \colon (\A_K^1)^M \to \A_K^1$ 
induced by the morphism of $K$-algebras 
\begin{align*}
K[t] & \to K[X_g: g \in M] \\
t & \mapsto \sum_{u\in \N_c^G} \alpha(u) \overline{X}^u. \numberthis \label{e:alg-nr-1}
\end{align*}
Then it is clear from \eqref{e:alg-nr} and \eqref{e:alg-nr-1} that 
$\sigma_\alpha \in CA_{alg}(G, \A^1, K)$ as it admits $\supp(\alpha)$ as a  minimal memory set 
and the associated local defining map $\mu_\alpha$ is induced by $f$.     
\par
Again, from \eqref{e:alg-nr-1} and the definition of $\supp(\alpha)$, 
we see easily that $\supp(\alpha)$ is the minimal memory set of $\tau$. 
Now suppose that $\sigma \in CA_{alg}(G, \A^1, K)$. 
Then $\sigma$ has a memory set $N\subset G$ 
and the associated local defining map $K^N \to K$ is induced by a polynomial 
$\beta=\sum_{u\in \N_c^G} \beta(u) \overline{X}^u \in K[X_g: g\in N]$ as in \eqref{e:alg-nr-1}. 
It is clear that $\sigma=\sigma_\beta$. 
\end{proof}

A commutative ring $R$ is said to have Property $(P)$ if the polynomial ring $R[T]$ 
has no nonzero polynomial that is identically zero  as a function $R \to R$. 
Any reduced infinite (commutative) ring without idempotents has Property $(P)$ (e.g. any infinite field). 
By an easy classical induction argument on the number of variables, we have: 

\begin{lemma}
\label{l:zero-poly}
If a commutative ring $R$ satisfies Property $(P)$ then for all $n\in \N$, 
the polynomial ring $R[T_1, \dots, T_n]$ has no nonzero polynomial that is identically zero on $R^n$. 
\end{lemma}

\begin{theorem}
\label{t:iso-nr-ca}
Suppose that $K$ is a commutative ring that satisfies $(P)$.  
Then we have an isomorphism of near rings $\Phi \colon R(K,G) \to CA_{alg}(G,\A^1, K)$ 
given by $\Phi(\alpha) = \sigma_\alpha$ for all $\alpha \in R(K,G)$.  
\end{theorem}

\begin{proof}
We verify first that $\Phi$ is injective. 
Indeed, suppose that $\sigma_\alpha=\sigma_\beta$ for some $\alpha, \beta \in R(K,G)$. 
Let $M= \supp(\alpha) \cup \supp(\beta) \subset G$ be a finite subset. 
Then we have two polynomials $\mu_\alpha$, $\mu_\beta$ in $K[X_g : g \in M]$ 
corresponding to the local defining maps of $\sigma_\alpha$ and $\sigma_\beta$ respectively. 
Since $\sigma_\alpha=\sigma_\beta$ on $K^G$, it follows that $\mu_\alpha=\mu_\beta$ as functions $K^M \to K$. 
Lemma \ref{l:zero-poly} implies that $\mu_\alpha=\mu_\beta$ as polynomials in $K[X_g: g\in M]$. 
We deduce that $\alpha=\beta$ and thus $\Phi$ is injective. 
The map $\Phi$ is surjective by Lemma \ref{l:map-phi}. 
\par
Let $r\in K$ and $\alpha, \beta \in R(K,G)$, the commutativity of $K$ implies that 
\[
\sigma_{\alpha + \beta}= \sigma_\alpha + \sigma_\beta, \quad \text{and} \quad \sigma_{r\alpha} = r \sigma_\alpha. 
\] 
It follows that $\Phi(\alpha + \beta)= \Phi(\alpha) + \Phi(\beta)$. 
We verify now that $\Phi$ is a homomorphism of near rings. 
Since $\sigma_\alpha \circ \sigma_\beta$ and $\sigma_{\alpha \star \beta}$ 
are cellular automata (Lemma \ref{l:map-phi}), they are $G$-equivariant. 
Thus to show that they are equal, it suffices to show that 
$(\sigma_\alpha \circ \sigma_\beta)(x)(1_G)= \sigma_{\alpha \star \beta}(x)(1_G)$ for every $x\in K^G$. 
Denote $y= \sigma_\beta(x) \in K^G$, we find that  
\begin{align*}
\sigma_\alpha (\sigma_\beta(x)))(1_G) = \sigma_\alpha(y)(1_G) & \coloneqq \sum_{u \in \N^G_c} \alpha(u) y^{u}\\
&= \sum_{u \in \N^G_c} \alpha(u) \prod_{g \in G} y(g)^{u(g)} \\
&= \sum_{u \in \N^G_c} \alpha(u) \prod_{g \in G} \left(  \sum_{v \in \N^G_c} \beta(v) x^{gv} \right)^{u(g)} \\
&= \sigma_{\alpha \star \beta}(x)(1_G)  
\end{align*}
The last equality follows clearly from the definitions \eqref{e:def-nr} and \eqref{e:alg-nr}. 
Therefore, $\sigma_\alpha \circ \sigma_\beta = \sigma_{\alpha \star \beta}$ 
and $\Phi(\alpha \star \beta)= \Phi(\alpha) \circ \Psi(\beta)$ for all $\alpha, \beta \in R(K,G)$. 
As $\Psi(1_{R(K,G)})=\Id_{K^G}$, we conclude that $\Phi$ is an isomorphism of near rings. 
\end{proof}

\begin{remark}
Let $CA_{lin}(G, \A^1, K)$ denote the ring of linear cellular automata over $K^G$. 
Suppose that $K$ is an infinite field.
If $\car(K) >0$, we have the following commutative diagram: 
\[
\begin{tikzcd}
& K[G] \arrow[d,"\simeq"]   \arrow[r] &  K[t;F][G] \arrow[d, "\simeq"]  \arrow[r] & R(K,G) \arrow[d, "\simeq"]  \\
& CA_{lin}(G, \A^1,K) \arrow[r] & CA_{algr}(G, \A^1, K) \arrow[r] & CA_{alg}(G,\A^1,K) 
\end{tikzcd}
\]
where the horizontal maps are natural inclusions. 
When $\car(K)=0$, we have the same diagram without the middle column.  
\end{remark}

\subsection{On generalized Kaplansky's conjectures}

Since the group ring $K[G]$ is canonically contained in the near ring $R(K,G)$ 
by Proposition \ref{p:embed-ring}, it is natural to extend Kaplansky's conjectures to $R(K,G)$ as follows. 

\begin{conjecture}[Generalized Kaplansky's conjectures]
\label{c:conjecture}
Let $K$ be a field and let $G$ be a group.  
Consider the near ring $(R(K,G), +, \star)$. 
If $G$ is torsion free then: 
\begin{enumerate} [\rm -]
\item
(Unit conjecture) All one-sided inverses of $R(K,G)$ are trivial units, i.e., 
if $\alpha, \beta \in R(K,G)$ and $\alpha \star \beta =X_{1_G}$ then $\alpha=aX_g-ab$ and 
$\beta= a^{-1}X_{g^{-1}}+b$ for some $g\in G$, $a\in K^*$, and $b\in K$; 
\item
(Zero-divisor conjecture) $R(K,G)$ does not contain nonzero divisor, i.e., 
if $\alpha, \beta \in R(K,G)$ and $\alpha \star \beta= 0$ then $\alpha=0$ or $\beta$ is a constant;
\item
(Nilpotent conjecture) $R(K,G)$ does not contain nonzero nilpotent, i.e., 
if $\alpha \in R(K,G)$ and $\alpha^{(\star n)}=0$ for some $n \geq 1$ then $\alpha=0$;
\item
(Idempotent conjecture) The idempotents of $R(K,G)$ are trivial, i.e., 
if $\alpha \in R(K,G)$ and $\alpha \star \alpha =\alpha$ then $\alpha=X_{1_G}+d$ where $d\in K$, $2d=0$ or 
$\alpha$ is a constant. 
\end{enumerate}
For an arbitrary group $G$, we have: 
\begin{enumerate}  [\rm -]
\item
(Direct finiteness conjecture) All one-sided inverses of $R(K,G)$ are units, i.e., 
if $\alpha, \beta \in R(K,G)$ and $\alpha \star \beta= X_{1_G}$ then $\beta \star \alpha=X_{1_G}$. 
\end{enumerate}
 \end{conjecture}
  
\begin{remark}
It is a difficult result that for any field $K$ and any group $G$, $K[G]$ is reduced (no nilpotent elements) 
if and only if $K[G]$ is a domain (no zero divisors) (cf. \cite[Chapter~2.\S 6]{lam}). 
However, this is false for the near ring $R(K,G)$. 
In fact, $R(K,G)$ is never a domain for any ring $K$ and any group $G$ 
since every constant is a divisor of $0$.  
To see this, let $\alpha \in R(K,G) \setminus \{0\}$ whose sum of coefficients 
$s(\alpha)=0$, e.g., $\alpha=X_g-1$ for any $g\in G$.  
Then it follows that $\alpha \star 1=s(\alpha)=0$. 
Hence, constants in $R(K,G)$ should be viewed as trivial right zero divisors. 
This explains the necessity of the modification in the generalized conjectures. 
However, we suspect that the nilpotent conjecture holds: if $\alpha \in R(K,G)$ 
and $\alpha^{(\star n)}=0$ for some $n \geq 1$ then $\alpha = 0$. 
\end{remark}
 
As an application of Theorem \ref{t:inverse-also-near-ring} and Theorem \ref{t:iso-nr-ca}, 
we establish first the following direct finiteness and stability properties of units of $K[G]$ 
and $K[t;F][G]$ in $R(K,G)$. 

\begin{theorem}
Let $K$ be a field with an algebraic closure $\overline{K}$ and let $G$ be a sofic group. 
Let $\alpha, \beta \in R(\overline{K},G)$ be such that $\alpha \star \beta =X_{1_G}$. 
Then the following hold. 
\begin{enumerate} [\rm (i)]
\item
if $G$ is locally residually finite and $K$ is uncountable then $\beta \star \alpha =X_{1_G}$;
\item
if $\beta \in K[G]$ then $\beta \star \alpha =X_{1_G}$ and $\alpha \in K[G]$; 
\item
if $\car(K)=p>0$ and $\beta \in \overline{K}[t;F][G]$ 
then $\beta \star \alpha =X_{1_G}$ and $\alpha \in \overline{K}[t;F][G]$. 
\end{enumerate}
\end{theorem}

\begin{proof} 
Recall the isomorphism $\Phi \colon R(K,G) \to CA_{alg}(G,\A^1, K)$ in Theorem \ref{t:iso-nr-ca}. 
Let $\tau=\Phi(\beta) \in CA_{algr}(G, X, \overline{K})$ and $\sigma=\Phi(\alpha) \in CA_{alg}(G,X,\overline{K})$. 
Since $\alpha \star \beta =X_{1_G}$, it follows that $\sigma \circ \tau =\Id_{\overline{K}^G}$. 
In particular, $\tau$ is injective. 
\par
Suppose that $G$ is locally residually finite. 
Then since $\tau$ is injective and $\overline{K}$ is uncountable, 
it follows from \cite[Theorem~1.2]{cscp-alg-ca} that $\tau$ is bijective. 
Hence, we have $\tau \circ \sigma=\Id_{\overline{K}^G}$. 
This proves (i) since $\beta \star \alpha= \Phi(\Id_{\overline{K}^G})=X_{1_G}$. 
\par 
Suppose that $\car(K)=p>0$ and $\beta \in \overline{K}[t;F][G]$. 
Let $X=\G_a=\Spec \overline{K}[T]$ be the additive group over $\overline{K}$ then 
$\End_{\overline{K}}(X)=\overline{K}[t;F]$ is the endomorphism ring of $X$. 
Then Theorem \ref{t:inverse-also-near-ring} implies that 
$\tau \circ \sigma=\Id_{\overline{K}^G}$ and $\sigma \in CA_{algr}(G,X,K)$. 
It follows immediately that $\beta \star \alpha =X_{1_G}$ and $\alpha \in \overline{K}[t;F][G]$. 
This proves (iii). 
\par
Now suppose that $\beta \in K[G]$. 
The proof for (ii) is similar but instead of using Theorem \ref{t:inverse-also-near-ring}, 
we can apply directly a similar result for linear cellular automata 
\cite[Theorem~1.2, 3.1]{csc-sofic-linear} to conclude. 
\end{proof}

\subsection{Generalized Kaplansky's conjectures for orderable groups} 

A group $G$ is called \emph{orderable} if it admits a strict total ordering $<$ 
which is left and right invariant, i.e., for all $g, h, k \in G$, $g < h$ implies $fg < fh$ and $gf<hf$. 
The class of orderable groups is closed under taking subgroups and direct products. 
\par
Suppose that $(G,<)$ is an orderable group. 
Then we have several induced well-orders on the set of monomials in $K[X_g: g \in G]$. 
Note that to give such an order is equivalent to give an order on $\N^G_c$. 
In what follows, we choose to work with the lexicographic order on $\N^G_c$ 
as well as the set of monomials in $K[X_g: g \in G]$. 
Using the definition of $\N^G_c$ and the multiplication law $\star$ on $\N^G_c$ 
and $R(K,G)$, the proof of the following lemma is elementary and straightforward. 
\begin{lemma}
\label{l:order-unit}
Suppose that $(G,<)$ is an orderable group. 
Suppose that $u, v, w \in \N^G_c$ and $g\in G$. 
Equip $\N^G_c$ with the induced lexicographic order also denoted  $<$. 
Then we have 
\begin{enumerate} [\rm (i)]
\item 
if $w\neq 0$ then $0< w$;
\item
if $w\neq 0$ then $u < u+w$; 
\item
$u <v$ if and only if $gu < gv$ if and only if $u+w < v+w$;
\item
if $u< v$ and $w \neq 0$ then $u\star w < v \star w$ and $w \star u < w \star v$. 
\end{enumerate}
Moreover, let $\alpha= \alpha(u)\overline{X}^u+ (\text{lower terms})$ and 
$\beta= \beta(v)\overline{X}^v+ (\text{lower terms})$ be elements in $R(K,G)$ where $K$ is a domain. 
Then we have 
\begin{align*}
\alpha \star \beta 
& =  \alpha(u) \beta(v)^{s(u)} \overline{X}^{u \star v}+ (\text{lower terms}),
\end{align*}
where $s(u) \coloneqq \sum_{g\in G}u(g) \in \N$. 
\end{lemma}

\begin{proof}
The points (i), (ii), (iii) are obvious. 
To check (vi), suppose that $g_1 < \dots < g_m$ (where $m\geq 1$ 
by (i) and $u<v$) are the support of $v$. 
If $u(g_m)=v(g_m)$ then since $(\N^G_c, \star)\simeq \N[G]$ is left and right distributive, 
we only need to check $(u - u(g_m)\mathbf{1}_{g_m}) \star w < (v - v(g_m)\mathbf{1}_{g_m}) \star w$ 
and $w \star (u - u(g_m)\mathbf{1}_{g_m}) < w \star (v - v(g_m)\mathbf{1}_{g_m})$. 
Up to replacing $u,v$ by $u-u(g_m)\mathbf{1}_{g_m}$ and $v-v(g_m)\mathbf{1}_{g_m}$ 
respectively and so on if necessary, we can thus suppose that $u(g_m) < v(g_m)$. 
Let $h\in G$ be the maximal element in $\supp(w)$. 
Then clearly 
\[
u\star w (g_mh)= u(g_m)w(h)  < v \star w(g_mh)=v(g_m)w(h)
\]
and thus $u\star w < v \star w$. 
Similarly, $w\star u < v \star w$ and this proves (iv). 
The proof of the last statement is similar and is omitted. 
We only notice that $ \alpha(u) \beta(v)^{s(u)}\neq 0$ since $K$ is a domain. 
\end{proof}

\begin{theorem}
\label{t:kap-unit}
Suppose that $G$ is an orderable group and $K$ is an integral domain. 
Then the near ring $R(K,G)$ satisfies the generalized Kaplansky unit conjecture.   
\end{theorem}

\begin{proof}
Let $\alpha , \beta \in R(K,G)$ be such that $\alpha \star \beta=X_{1_G}$. 
Then clearly $\alpha, \beta$ are not constant since $1 \star \beta =1$ and $\alpha \star 1$ are constant. 
\par
Therefore, we can write $\alpha= \alpha(u)\overline{X}^u+ (\text{lower terms})$ and 
$\beta= \beta(v)\overline{X}^v+ (\text{lower terms})$ 
where $u, v \in \N^G_c$ are nonzero functions and $\alpha(u), \beta(v) \neq 0$. 
By Lemma \ref{l:order-unit}, we have 
\[
X_{1_G}=\alpha \star \beta 
=  \alpha(u) \beta(v)^{s(u)} \overline{X}^{u \star v}+ (\text{lower terms}).
\]
As $\alpha(u) \beta(v)^{s(u)}\neq 0$, Lemma \ref{l:conv-prd} implies immediately that $u \star v= \mathbf{1}_{1_G}$. 
Since $\supp(u\star v)=\supp(u) \supp(v)$, we deduce that $\supp(u)$ and $\supp(v)$ are singleton. 
Thus, it follows that $u=\mathbf{1}_g$ and $v=\mathbf{1}_{g^{-1}}$ for some $g\in G$. 
Therefore, $s(u)=1$ and thus $\alpha(u) \beta(v)=1$. 
\par
Consider now the reversed order $>$ on $G$ and the corresponding induced monomial order on $R(K,G)$. 
By the same argument, we deduce that $\alpha$ (resp. $\beta$) 
does not have any non-constant monomials other than $X_g$ (resp. $X_{g^{-1}}$). 
Indeed, suppose that there exist such non-constant monomials. 
Then the product of two leading terms of $\alpha$ and $\beta$ with respect to $>$ is 
a non-constant monomial distinct from $X_{1_G}$ which is in fact $< X_{1_G}$ by Lemma \ref{l:order-unit}.(iv). 
This follows from the observation that for all $h \in \supp(\alpha - \alpha(u)X_g)$, we have $h < g$
and similarly, $h< g^{-1}$ for all  
$h \in \supp(\beta - \beta(u)X_{g^{-1}})$. 
But the product of these two leading terms is the leading term of 
$\alpha \star \beta$ with respect to $>$ by Lemma \ref{l:order-unit}, 
we arrive at a contradiction. 
Hence, using the equality $\alpha\star \beta=X_{1_G}$, 
it follows easily that $\alpha = a X_g -ab$ and $\beta=a^{-1} X_{g^{-1}} +b$ for some $a=\alpha(u) \in  K^*$ and $b\in K$. 
\end{proof}

As an application, we obtain the following nontrivial result saying that 
all injective algebraic cellular automata $\C^G \to \C^G$ where $G$ is in a certain class of groups are trivial. 
\begin{corollary}
\label{c:classification}
Let $G$ be a locally residually finite and orderable group (e.g. $\Z^d$ and free groups) 
and let $K$ be an uncountable algebraically closed field of characteristic $0$. 
Suppose that $\tau \in CA_{alg}(G, \A^1, K)$ is injective. 
Then there exists $g\in G$, $a \in K^*$, and $b\in K$
such that for all $x \in K^G$ and $h \in G$, we have 
\[
\tau(x)(h)= a x(g^{-1}h) +b. 
\]
\end{corollary}

\begin{proof}
Via the isomorphism $\Phi$ in Theorem \ref{t:iso-nr-ca}, $\tau$ is given by 
an element $\beta \in R(K, G)$ so that $\tau=\Phi(\beta)=\sigma_\beta$. 
By \cite[Corollary~1.5]{cscp-alg-ca} (which is true for the field $K$), 
$\tau$ is in fact invertible since it is injective by hypothesis. 
Let $\sigma \in CA_{alg}(G,\A^1, K)$ be its inverse and $\alpha=\Phi^{-1}(\sigma) \in R(K, G)$.  
It follows again from Theorem \ref{t:iso-nr-ca} that $\alpha \star \beta=X_{1_G}$. 
We can thus apply Theorem \ref{t:kap-unit} to conclude that 
$\beta = a X_g +b $ for some $g\in G$, $a \in K^*$, and $b\in K$. 
By the isomorphism $\Phi$, this means exactly that $\tau(x)(h)= a x(g^{-1}h) +b$ 
for all $x \in K^G$ and $h\in G$. 
\end{proof}

\begin{remark}
The proof of the above result is a combination of a point-wise method (maps $\C^G \to \C^G$) 
and a point-wise-free one (representation of such maps as elements of $R(\C,G)$) 
via the isomorphism $\Phi$ given in Theorem \ref{t:iso-nr-ca}. 
\end{remark}

\begin{proof}[Proof of Theorem \ref{t:zero-kap}] 
Suppose that $\alpha \star \beta =0$, $\alpha \neq 0$ and $\beta$ is nonconstant. 
Then $\alpha$ is also nonconstant since otherwise $\alpha=c \in L \setminus \{0\}$ and $\alpha \star \beta=c\neq 0$. 
It follows that $\supp(\alpha)$ and $\supp(\beta)$ are nonempty. 
Let $u,v\in \N^G_c \setminus \{0\}$ be respectively the leading exponents of 
$\alpha$ and $\beta$ with respect to the lexicographic monomial order induced by the order $<$ of $G$. 
Then Lemma \ref{l:order-unit} says that the leading exponent of $\alpha \star \beta$ is $u\star v$. 
Note that $0 < u \star v$ since $0 < u, v$ (cf. Lemma \ref{l:order-unit}.(iv)). 
Thus $\alpha \star \beta$ cannot be constant which contradicts the hypothesis $\alpha \star \beta =0$. 
This proves (i). 
\par
Now suppose that $P(\alpha)=aX_{1_G}$ for some 
$\alpha \in R(L,G)$, $a\in L$ and $P\in L[x]$ a polynomial of degree $m\geq 1$. 
Suppose that $\alpha$ is nonconstant and thus $\supp(\alpha)$ is nonempty 
\par
We first show that $\supp(\alpha)=\{1_G\}$. 
Suppose on the contrary that there exists $g\in \supp(\alpha) \setminus \{1_G\}$. 
Then $g<1_G$ or $1_G<g$. 
Consider now the case $1_G<g$. 
The leading exponent $u\in \N^G_c\setminus \{0\}$ of $\alpha$ exists 
and its maximal element is then at least $g$. 
It follows that $\mathbf{1}_{1_G} < \mathbf{1}_{g} \leq u$. 
By applying inductively Lemma \ref{l:order-unit}, 
$u^{(\star n)}\coloneqq u \star \dots \star u \text{ (n-times)}$ is the leading 
exponent of $\alpha^{(\star n)}$ for every $n \geq 1$. 
Since $\mathbf{1}_{1_G} < \mathbf{1}_{g} \leq u$, Lemma \ref{l:order-unit}.(iv) implies 
that $\mathbf{1}_{1_G} < u^{(\star n)}$ for all $n\geq 1$. 
In particular $u^{(\star n)}\neq 0$ 
and the same lemma yields 
$u^{(\star n)}= u^{(\star n)} \star \mathbf{1}_{1_G} < u^{(\star n)} \star u= u^{(\star (n+1))}$ for all $n\geq 1$. 
We deduce that $u^{(\star m)}$ is the leading exponent of $P(\alpha)$. 
But then $P(\alpha)$ cannot be $aX_{1_G}$ 
since $\mathbf{1}_{1_G}< u^{(\star m)}$, which is a contradiction. 
The same argument applies to the case $g<1_G$ with the reversed order of $G$. 
Therefore, $\supp(\alpha)=\{1_G\}$ as claimed. 
\par
Hence, there exist $p \in \N_{\geq 1}$ and $c_0, \dots, c_p \in L$ with $c_p\neq 0$ such that 
\[
\alpha=c_pX_{1_G}^p+\dots+ c_1X_{1_G}+c_0.
\]
If $p\geq 2$ then $P(\alpha)$ is a polynomial of degree 
$p^m \geq 2$ in $X_{1_G}$ and cannot be $aX_{1_G}$. 
Thus, $p=1$ and we can conclude that $\alpha=cX_{1_G}+d$ for some $c,d \in L$.  
Applying this to the case $P(x)=x^2-x$, the last assertion of Theorem \ref{t:zero-kap}.(ii) 
follows from a direct calculation. 
\end{proof}

\section{Some counter-examples}
We begin with an example showing that Theorem \ref{t:sofic} 
and Theorem \ref{t:unit-kap}.(ii) are false in positive characteristic. 

\begin{example}
\label{ex:main-frobenius}
Suppose that $G=\{ \sbullet[.9] \}$ is a trivial group 
and $K=\overline{F}_p$ where $p>0$ is a prime number. 
Let $F \colon \G_m \to \G_m$ be the Frobenius endomorphism 
of the multiplicative group $\G_m$ over $K$. 
Then the induced algebraic group cellular automaton 
$\tau \in CA_{algr}(G,\G_m, K)$ is given by $\tau \colon K^* \to K^*$, $x\to x^p$. 
Clearly, $\tau$ is bijective but its inverse map 
$\tau^{-1}\colon K^* \to K^*$, $x\to x^{1/p}$ is not algebraic 
hence $\tau^{-1} \notin CA_{alg}(G,\G_m, K)$.   
\end{example}

The following example shows that we cannot replace 
\emph{algebraic group cellular automaton} 
by a more general \emph{algebraic cellular automaton} in 
the hypotheses of 
Theorem~\ref{t:goe-grp} and Proposition~\ref{p:sur-dim} . 
\label{s:counter}\begin{example} (cf.~\cite[Example~8.3]{cscp-alg-goe})
Let $G$ be an amenable group with a F\o lner net $\FF$ 
and let $X$ be a connected $K$-algebraic group. 
Let $A\coloneqq X(K)$ and let $\tau \colon A^G \to A^G$ be an algebraic cellular automata 
with a memory set $M=\{1_G\}$. 
Then the local defining map $\mu \colon A \to A$ is induced 
by a morphism of $K$-algebraic varieties $f \colon X \to X$. 
Note that $f$ is not necessarily a homomorphism of algebraic groups. 
Theorem~\ref{t:goe-grp} and Proposition~\ref{p:sur-dim} do not apply in this case. 
In fact, if $f$ is a dominant nonsurjective morphism, then we see clearly 
that $\mdim_\FF(\tau(A^G))=\dim(A)$ while $\tau$ is not surjective. 
Moreover, it is also immediate that $\tau$ 
is $(\sbullet[.9])$- and $( \sbullet[.9] \sbullet[.9]) $-pre-injective. 
For example, we can take $X=\A^2=\Spec(K[x,y])$ and 
$f\colon \A^2 \to \A^2$ to be induced by the morphism of $K$-algebras
\begin{align*}
K[x,y] & \to K[x,y] \\
(x,y) & \mapsto (x,xy). 
\end{align*}
Then $f(\A^2)=\A^2 \setminus (\{x=0\} \setminus \{(0,0)\})$ and thus $f$ is dominant but not surjective. 
\end{example}

The next example shows that we cannot omit the hypothesis that $K$ is algebraically closed in 
Theorem \ref{t:sofic}, Theorem~\ref{t:surjunctive-res}, Theorem~\ref{t:goe-grp}, and Corollary~\ref{c:myhill}. 

\begin{example}
Let $E$ be an elliptic curve over $\Q$. 
Then $E$ is a connected and projective $\Q$-algebraic group. 
By Mordell-Weil's theorem, $E(\Q)$ is a finitely generated abelian group.
Consequently, $E(\Q) = F \oplus T$, where $F$ is a free abelian group 
of finite rank $r \geq 0$ and $T$ is a finite abelian group. 
Mazur's theorem asserts that
the group $T$ is either cyclic of order $n \in \{1,\dots, 12\}\setminus \{11\}$ or 
 isomorphic to $\Z /2m \Z \times \Z/ 2\Z$ with $m \in \{1,\dots,4\}$. 
Suppose that $r \geq 1$ 
(e.g. $E$ is the elliptic  curve with equation $y^2=x^3+4x+9$
 for which $r = 2$).
Let $p\geq 11$ be a prime number. 
Let $f \coloneqq [p]_E \colon E \to E$ be the multiplication-by-$p$ map on $E$. 
Then $f$ is a homomorphism of $\Q$-algebraic groups inducing the multiplication by $p$ on $E(\Q)$. 
Since $p\geq 11$, the map 
$f$ induces a group automorphism of $T$.
On the free part $F$, 
it induces a group  endomorphism that is clearly injective but not surjective.
Consequently, $f$ induces a group endomorphism of $E(\Q)$ that is injective but not surjective.
\par
Now let $G$ be any group, $A\coloneqq E(\Q)$, 
and let $\tau\colon A^G \to A^G$ denote the algebraic group cellular automaton over $(G,E,\Q)$ 
with memory set $M=\{1_G\}$ and associated local defining map $\mu=f^{(\Q)} \colon A \to A$. 
The above paragraph shows that $\tau$ is injective (and thus pre-injective) but not surjective.
\end{example}
 
The following classical example shows that Theorem \ref{t:classification} 
fails as soon as we replace $\A^1$ by any higher dimensional affine space $\A^n$ ($n\geq 2$). 

\begin{example}
\label{ex:torsion-free-injective}
Let $n\geq 2$. 
Choose any linear automorphism $f$ of the vector space $\C^n$ such that $f$ is a not a homothety. 
Then for any group $G$, the algebraic cellular automaton $f^G \colon (\C^n)^G \to (\C^n)^G$ is injective. 
However, there do not exist $g\in G$, $a\in \C^*$ and $b\in \C$ 
such that $f^G(x)(h)=ax(g^{-1}h)+b$ for all $x\in \C^G$ and $h\in G$.  
\par
We consider now a non degenerate case. 
Let $G=\Z$ and $M=\{0,1\}\subset G$. 
Let $\tau, \sigma  \colon (\C^2)^G \to (\C^2)^G$ be linear cellular automata induced 
respectively by the local defining maps $\mu, \nu \colon (\C^2)^M \to \C^2$ given by: 
\[
\mu((x,y); (z,t)) = (x+t, y), 
\]
and 
\[
\nu((x,y); (z,t)) = (x-t, y).
\]
We check easily that $\sigma \circ \tau = \tau \circ \sigma = \Id_{(\C^2)^G}$ 
and that $M$ is the minimal memory set of $\sigma, \tau$. 
Similar examples show that we cannot even control the size of the minimal memory set 
of an injective or even bijective linear cellular automaton. 
\end{example}

\bibliographystyle{siam}

\end{document}